%% file: thesis.tex
\documentclass{ucbthesis}
\usepackage[backend=bibtex]{biblatex}
\usepackage{amsmath,amsthm,amssymb}
\usepackage[all]{xy}
\newtheorem{thm}{Theorem}[section]

\newtheorem{prop}[thm]{Proposition}
\newtheorem{lem}[thm]{Lemma}
\theoremstyle{definition}
\newtheorem{defi}{Definition}
\theoremstyle{remark}
\newtheorem{rmk}{Remark}
\newcommand{\isomto}{\stackrel{\sim}{\longrightarrow}}
\newcommand{\vect}[2]{\begin{pmatrix} #1\\#2 \end{pmatrix}}
\newcommand{\define}[1]{\textbf{#1}}
\newcommand{\cat}[1]{\mathbf{#1}}
\newcommand{\calD}{\mathcal{D}}
\newcommand{\calF}{\mathcal{F}}
\newcommand{\calL}{\mathcal{L}}
\newcommand{\calO}{\mathcal{O}}
\newcommand{\frakg}{\mathfrak{g}}
\newcommand{\frakX}{\mathfrak{X}}
\newcommand{\oc}{\mathrm{oc}}
\newcommand{\NN}{\mathbb{N}}
\newcommand{\ZZ}{\mathbb{Z}}

\newcommand{\RR}{\mathbb{R}}
\newcommand{\CC}{\mathbb{C}}
\newcommand{\KK}{\mathbb{K}}
\newcommand{\TT}{\mathbb{T}}
\DeclareMathOperator{\id}{id}
\DeclareMathOperator{\im}{im}
\DeclareMathOperator{\pr}{pr}
\DeclareMathOperator{\spann}{span}
\DeclareMathOperator{\Hom}{Hom}
\DeclareMathOperator{\Der}{Der}
\bibliography{references}
\begin{document}

\title{The Tangent Functor Monad and Foliations}
\author{Beno{\^{\i}}t Michel Jubin}
\degreesemester{Spring}
\degreeyear{2012}
\degree{Doctor of Philosophy}
\chair{Professor Alan Weinstein}
\othermembers{Professor Robert Bryant\\ Professor Robert Littlejohn}
\numberofmembers{3}
\field{Mathematics}
\campus{Berkeley}

\copyrightpage

\include{abstract}

\begin{frontmatter}

\tableofcontents

\begin{acknowledgements}
First of all, I wish to thank my advisor Alan Weinstein for the guidance he has provided me during my graduate studies.
His broad interests and his geometric insight have been very valuable to me, and his advice and answers to my many questions were of great help.
It was an honor and a very rewarding experience to work under his supervision.

I would also like to thank Robert Bryant and Robert Littlejohn for accepting to be on my dissertation committee and for their interesting comments on earlier versions of my work.

My gratitude also goes to my earlier professors back in France, who fostered my interest in differential geometry and who made my doctoral studies in Berkeley possible.
I am especially indebted to Pierre Pansu, Jean-Pierre Bourguignon and Denis Auroux (whom I ended up meeting again here in Berkeley!).

Thank you to my academic siblings, Aaron, Hanh, Santiago and Sobhan, for the nice times we had, in Berkeley and Paris, and for teaching me some mathematics in neighboring areas --- this is another benefit of working with Alan: the breadth of his interests spreads to his students.
Thanks also to all the students and post-docs who joined Alan's group from a few weeks to a whole year, and shared with us their work and enthusiasm during our weekly group meetings: Camilo, Leandro, Ludovic, Mitsuhiro, Chiara, Sean, Henry... and Arnaud for his guest appearance.
I also benefited from the visits of Christian Blohmann, Benoit Dherin, Marco Fernandes and Slobodan Simi{\'c}.

Thanks to my officemates in Evans Hall, Jacob, Erica, Boris, Jeff and Chul-Hee, for providing a nice studying environment, and to them and many more students and friends from the department for contributing to the good atmosphere here. This good atmosphere also owes a lot to the staff of the mathematics department, who I sincerely thank for their help.

Finally, these acknowledgements would not be complete without my warmest thanks to my parents.
Thank you for your constant support during these doctoral years, as well as all the years before.
\end{acknowledgements}
\end{frontmatter}

\pagestyle{headings}

\include{chap1}
\include{chap2}

\include{chap3}

\include{chap4}

\include{chap5}

\printbibliography

\end{document}

%% file: abstract.tex

\begin{abstract}


In category theory, monads, which are monoid objects on endofunctors, play a central role closely related to adjunctions.
Monads have been studied mostly in algebraic situations.
In this dissertation, we study this concept in some categories of smooth manifolds.
Namely, the tangent functor in the category of smooth manifolds is the functor part of a unique monad, which is the main character of this dissertation.

After its construction and the study of uniqueness properties in related categories, we study its algebras, which are to this monad what representations are to a group. We give some examples of algebras, and general conditions that they should satisfy.
We characterize them in the category of affine manifolds.
We also study an analog of the tangent functor monad and its algebras in algebraic geometry.

We then prove our main theorem: algebras over the tangent functor monad induce foliations on the manifold on which they are defined.
This result links the study of these algebras to the study of foliations.
A natural question is then to characterize foliations which arise this way.
We give some restrictions in terms of the holonomy of such foliations.
Finally, we study in greater detail algebras on surfaces, where they take a very simple form, and we nearly characterize them.

\end{abstract}

%% file: chap1.tex
\chapter{Introduction}
\label{chap:intro}

This dissertation studies a certain monoidal structure, called a monad, on the tangent functor of some categories of manifolds. Monads play a central role in category and are related to adjunctions.
In this introductory chapter, first we briefly introduce monads and their algebras, then we give an outline of this dissertation, and we end with questions for future research, and a section giving the notations and conventions we use in the rest of the dissertation.

\section{Monads and their algebras}

A monad in a category is a monoid object in the monoidal category of endofunctors of that category (with monoidal product horizontal composition), that is, it is a triple $(T,\eta,\mu)$ where $T$ is an endofunctor of a category $\cat{C}$, and the unit $\eta \colon \id_{\cat{C}} \Rightarrow T$ and the multiplication $\mu \colon T^2 \Rightarrow T$ are natural transformations such that $\mu \circ T\mu = \mu \circ \mu T$ (associativity) and $\mu \circ T\eta = \mu \circ \eta T = \id_{\cat{C}}$ (unit laws) --- where concatenation means horizontal composition (see the section on notations and conventions).

In the same way that groups can be studied via their actions, monads can be studied via their modules, or algebras, which are generalizations of monoid actions. Namely, an algebra over a monad $(T,\eta,\mu)$ in a category $\cat{C}$ is a $\cat{C}$-morphism $h \colon TM \to M$ such that $h \circ \eta_M = \id_M$ (acting by the identity is like not acting) and $h \circ T h = h \circ \mu_M$ (acting successively by two elements is like acting by their product). Morphisms of algebras are defined in an obvious way (morphisms $f \colon M \to N$ making the obvious diagram commute), thus algebras over a given monad form a category.

A basic example of monad is the group monad in the category of sets. The endofunctor of this monad associates to each set the underlying set of the free group on that set. This monad comes from the adjunction between sets and groups given by the forgetful functor and the free group functor. This is a special case of a general construction which associates a monad to any adjunction. For the sake of concision, we give an explicit description of the semigroup monad. The endofunctor associates to a set of letters a set of (nonempty, finite) words in these letters. The unit associates to each set of letters the inclusion of letters (that is, the function which to a letter associates the word of length 1 consisting of that letter). The multiplication associates to each set of letters the operation of concatenation (from the set of sentences to the set of words, corresponding to ``ignoring white spaces''). 

The algebras over the semigroup monad are semigroups. Indeed, an algebra $h \colon TX \to X$ can be thought of as the simultaneous definition of all $n$-fold products, and the axioms for algebras translate exactly into the axioms for semigroups.
Actually, the category of algebras over the semigroup monad forms an adjunction with the category of sets which is isomorphic to the initial adjunction.
This behavior is typical of algebraic structures and their free objects.
To be more precise, one can define the notion of monadicity of an adjunction, and use Beck's monadicity theorem to prove that varieties of universal algebras and their free object functors form monadic adjunctions. See~\cite{maclane} for a detailed exposition.

\section{Outline of the dissertation}

We give an outline of this dissertation and its main results.
It is written in an informal language: rigorous definitions will be given in the section on notations and conventions at the end of this chapter and in the main text.
The following four subsections correspond to Chapters~2 to~5.

\subsection{The tangent functor monad and its algebras}

The tangent functor monad is a monad in the category of smooth manifolds. Its endofunctor is the tangent endofunctor. Its unit is the zero section $\zeta \colon \id_{\cat{Man}} \Rightarrow T$ and its multiplication $\mu = \tau T + T \tau \colon T^2 \Rightarrow T$ is the sum of the two projections from the second tangent bundle to the tangent bundle.
As we will see, natural transformations between the involved functors are scarce, and we can prove that the tangent functor monad is unique among monads on the tangent functor.
Also, there is no comonad on the tangent functor, and the cotangent functor seems to give no satisfactory notion of (co)monad on it.

An algebra over the tangent functor monad is a map $h \colon TM \to M$ where $M$ is a manifold, satisfying the two axioms $h(x,0)=x$ for all $x \in M$, and $h(Th(\xi)) = h \big( \tau_{TM}(\xi) + T\tau_M(\xi) \big)$ for all $\xi \in T^2M$. The latter axiom is more mysterious, and we study some of its consequences, in particular from its expression in charts.

Two basic examples of algebras are the tangent bundle projections $\tau_M \colon TM \to M$, and the free algebras $\mu_M \colon T^2M \to TM$.
We associate to an algebra what we call its rank, which is the rank of its fiberwise derivative along the zero section.
This rank is at most half the dimension of the base manifold, and the two basic examples give the extreme values of this rank: 0 for $\tau_M$, and $\frac12 \dim (TM)$ for $\mu_M$.
We actually prove that an algebra of rank 0 is a tangent bundle projection.

We also give several general methods of constructions of algebras: product algebras, subalgebras, quotient algebras and lifted algebras on covering spaces.
The latter construction is used twice in the following chapters to reduce classification problems to easier cases: it is applied in Chapter~\ref{chap:aff} to holonomy covering spaces, and in Chapter~\ref{chap:foli} to orientable covers of foliations.

The chapter ends with an analogous construction in algebraic geometry: the tangent functor comonad in the category of commutative rings, and its coalgebras.

\subsection{The affine case}

In this chapter, we characterize algebras in the affine category, where the problem is greatly simplified.
To do this, we first solve it for algebras on open subspaces of affine spaces. There, an affine algebra $h \colon TU \to U$ is of the form $h(x,v) = x +Av$ where $A$ is a linear endomorphism with $A^2=0$.

Then, we use the Auslander--Markus theorem, which asserts that an affine manifold has a covering space with trivial holonomy.
Therefore, by the general construction of the previous chapter, we can lift an algebra on a given affine manifold to an algebra on its holonomy covering space, where we use the previous result on domains of charts.

We also study ``semi-affine'' examples, that is, algebras $h \colon TU \to U$ on open subspaces of affine spaces of the form $h(x,v) \mapsto x + A_x v$ where $A$ is a $(1,1)$-tensor field on $U$. If such an $h$ is an algebra, then ($A^2 = 0$ and) the Nijenhuis tensor of $A$ has to vanish. In particular, for a semi-affine algebra of maximal rank, $\im A = \ker A$, so $A$ is an integrable almost tangent structure.

Since the affine category has fewer morphisms, being a natural transformation is less demanding, and indeed there are other monads and comonads on the tangent functor, that we classify.
They pair up to form bimonads and, in certain conditions, Hopf monads, and we give some examples of Hopf modules over these.

\subsection{The foliation induced by an algebra}

This chapter contains the main structural result of this dissertation: algebras induce foliations on their base manifolds.
The dimensions of the leaves of these foliations may vary. Such foliations are called singular, of \v{S}tefan foliations.

To prove this result, we first introduce the accessible sets $h(T_xM)$ of an algebra $h$, and we prove that they form a partition of $M$.
To show that they are the leaves of a foliation, we use results of \v{S}tefan and Sussmann on singular foliations and their tangent distributions, which generalize the classical Fr\"{o}benius integrability theorem.
\v{S}tefan associates to any smooth distribution its accessible sets, which form a foliation, and we use this result as follows.
We prove that the image of the fiberwise derivative of an algebra on the zero section is a smooth distribution, and that its accessible sets in the sense of \v{S}tefan are the accessible sets of the algebra.
Under mild conditions of tameness of the algebra, that smooth distribution is actually integrable: it is the tangent distribution to that foliation.

This result establishes a connection to foliation theory.
Roughly speaking, one can see an algebra as a foliation together with a way of wandering on the leaves.
This also raises a question: which foliations arise this way?
Such ``monadic'' foliations have leaves of dimension at most half the dimension of the manifold.
In addition to this obvious fact, we derive a necessary condition in terms of their holonomy: 1 has to be a common eigenvalue of their linear holonomy.
To prove this result, we have to establish a result with a control theoretic flavor: one can lift a path in a leaf to a path in the ``parameter space'' of that leaf, which is the tangent space of the manifold at the origin of the path.

We end the chapter on questions of orientability of foliations. We see how to use the orientable cover of a foliation and our general result on lifted algebras to reduce the study of regular algebras to the study of orientable ones.

\subsection{Algebras of rank 1 and algebras on surfaces}

Algebras of rank 0 were classified in Chapter~\ref{chap:tfmonad}. Here, we study the next simplest case of algebras of rank 1. Regular algebras of rank 1 assume a very simple form: they are given by following the flow of a vector field for a certain time. Explicitly, they are of the form $h(x,v) = \phi_{\alpha(x,v)}(x)$ where $\phi$ is the maximal flow of the vector field $X$ and $\alpha \in C^\infty(TM)$, the ``time function'', vanishes on the zero section.

The axioms for algebras translate into axioms on the time functions $\alpha$.
One axiom, $\alpha(x,v + \lambda X_x) = \alpha(x,v)$ for any $\lambda \in \RR$, shows that the time function is ``semibasic'', that is, is constant in the direction of the vector field, so can be considered as a transverse structure.
The other one, $\alpha \big( \phi_{\alpha(x,v)}(x), {\phi_{\alpha(x,v)}}'(x) \dot{x} \big) = \alpha(x,v+\dot{x}) - \alpha(x,v)$, is a kind of invariance property along the orbits.

If $\alpha$ is a 1-form (and $X$ is a vector field), then a function $h$ given by $h(x,v) = \phi_{\alpha_x(v)}(x)$ is an algebra if and only if $\alpha$ is basic, that is, semibasic ($i_X \alpha=0$) and invariant ($\calL_X \alpha = 0$).
This does not solve the whole classification problem, since we give an example of an algebra which cannot be given by following a flow for a time given by a 1-form.
However, in the case of surfaces, if the induced foliation has a non-vanishing basic 1-form, then it is possible to write the algebra as $h(x,v) = \phi_{\alpha_x(v)}(x)$ where $\alpha$ is a 1-form, necessarily basic.

Finally, we characterize morphisms between algebras of rank 1, which are akin to conjugation of dynamical systems: in the case of algebras given by 1-forms, an algebra morphism relates the corresponding vector fields (up to rescaling) and pullbacks the 1-form on the codomain to the 1-form on the domain.

\section{Further directions}

This dissertation raises several natural questions, with the ultimate goal of classifying algebras over the tangent monad, and understanding what they correspond to. We see from our main structure theorem that algebras are, roughly speaking, foliations with a way to wander on the leaves. An intermediate step is therefore to characterize monadic foliations.
The restrictions we have found, on the dimensions of the leaves and on the holonomy, are rather mild.
Some further conditions should be found, especially on transversality to certain submanifolds: if a submanifold is transverse to a monadic foliation, it bears certain $p$-plane fields, and this should give obstructions from algebraic topology.
On the other hand, we would like to construct other examples of monadic foliations which are ``far from linear''.

As we have seen in the case of rank 1, there is a kind of transverse structure (given in that case by the time function $\alpha$) of any monadic flow. The theory of transverse structures of foliations, most importantly of Riemannian foliations, has been greatly studied, and some of the techniques developed to study these foliations should prove useful in studying algebras.

Another natural question is to look at other categories, for instance the real or complex analytic theory, since we know that analytic foliation theory bears drastic changes from the smooth category. Also, we would like to pursue the study of coalgebras over the tangent functor comonad in the category of commutative rings (which we should extend to the category of schemes), and the study of Hopf modules in the category of affine manifolds.

The goal of classifying all algebras is probably hopeless, but it would be interesting to study algebras on 3-manifolds, which are of rank at most one. Also, algebras of maximal rank seem to bear some relation with almost tangent structures, and it would be interesting to investigate them.

\section{Notations and conventions}

Throughout this thesis, the letter $\KK$ will denote either the field of real numbers $\RR$ or the field of complex numbers $\CC$, and vector spaces will be over $\KK$.
If $E$ and $F$ are topological vector spaces, then $\calL(E,F)$ (resp. $\calL(E)$) denotes the set of \emph{continuous} linear maps from $E$ to $F$ (resp. to itself). 
A subspace $A$ of a topological vector space $E$ is said to split, or to be complemented, if its has a topological complement, that is, a subspace $B$ such that $+ \colon A \times B \to E$ is a homeomorphism.
In complete metrizable vector spaces (in particular, in Banach spaces), being complemented is equivalent to being closed with a closed complement (this is the open mapping theorem), and in Hilbert spaces, this is equivalent to being closed.
By the Hahn--Banach theorem, a finite dimensional subspace of a Hausdorff locally convex vector space (in particular, of a Banach space) splits.

Categories are denoted in boldface. A general category is denoted by $\cat{C}$. The collection of objects of $\cat{C}$ is denoted by $|\cat{C}|$. By the category of manifolds, denoted by $\cat{Man}$, we mean any notion of category of manifolds and maps with a convenient (in a naive sense) differential calculus and stable under the tangent functor. 
In Chapters~\ref{chap:aff} and~\ref{chap:foli}, we will restrict ourselves to smooth paracompact Hausdorff real or complex Banach manifolds and some subcategories stable under the tangent functor, for instance the categories $\cat{AffMan}$ and $\cat{CompAffMan}$ of affine manifolds (defined in Chapter~\ref{chap:aff}) and complete affine manifolds.
Although some results extend, we will not study in depth the category $\cat{BanAnaMan}$ of Banach analytic manifolds.
In Chapter~\ref{chap:surf}, we consider only real manifolds.

The category of vector bundles and vector bundle maps is denoted by $\cat{VecBun}$ and the category of fibered manifolds and fibered maps by $\cat{FibMan}$.

Finally, recall that natural transformations can be composed vertically and horizontally: if $F, G, H \colon \cat{C} \to \cat{D}$ are three functors and $\alpha \colon F \Rightarrow G$ and $\beta \colon G \Rightarrow H$ are natural transformations, then their \define{vertical composition} is denoted by $\beta \circ \alpha \colon F \Rightarrow H$ and is defined by $(\beta \circ \alpha)_M = \beta_M \circ \alpha_M$ for any object $M$ of $\cat{C}$. If $F', G' \colon \cat{D} \to \cat{E}$ are two other functors and $\gamma \colon F' \Rightarrow G'$, then the \define{horizontal composition} of $\alpha$ and $\gamma$ is denoted by $\gamma\alpha \colon F' \circ F \Rightarrow G' \circ G$ and is defined by $(\gamma\alpha)_M = \gamma_{G(M)} \circ F'(\alpha_M) = G'(\alpha_M) \circ \gamma_{F(M)}$ for any $M \in |\cat{C}|$. When a functor is horizontally composed with a natural transformation, we regard it as the identity natural transformation of itself. Exponential notation refers to vertical composition.

For the general theory of smooth manifolds, including infinite-dimensional Banach manifolds, we refer to~\cite{abra} or~\cite{lang}. As for notations, if $f \colon M \to N$ is a smooth map, its tangent map is denoted by $Tf \colon TM \to TN$ and its tangent map at a point $x \in M$ by $T_xf \in \calL(T_xM,T_{f(x)}N)$. If $M$ and $N$ are modeled respectively on the vector spaces $E$ and $F$ and we are given charts of $M$ at $x$ and of $N$ at $f(x)$, then we will write $f'(x) \in \calL(E,F)$ for the expression of $T_xf$ in these charts, the charts themselves generally remaining implicit. We will also make some identifications when the context is clear: for instance, if $(x,v) \in TM$ and charts are given, we will write $Tf(x,v) = \big( f(x), f'(x)v \big)$.
Finally, recall that a smooth map of locally constant finite rank is a subimmersion.

A chart of $M$ induces charts of $TM$ and $T^2M$. Therefore, if $\xi \in T^2M$, say $\xi \in T_{(x,v)}TM$, and a chart of $M$ is given, we will often write
\[
\xi = (x,v,\dot{x},\dot{v}) \in U \times E \times E \times E
\]
where $U$ is the domain of the chart.
We will also use the formula
\begin{equation*}
T^2f(x,v,\dot{x},\dot{v}) = \big(
f(x),
f'(x)v,
f'(x)\dot{x},
f'(x)\dot{v} + f''(x)(v,\dot{x})
\big)
\end{equation*}
which shows in particular that while $v$ and $\dot{x}$ have intrinsic meanings, $\dot{v}$ does not. Another example is the \define{canonical endomorphism} of $T^2M$, which is given in charts by $(x,v,\dot{x},\dot{v}) \mapsto (x,v,0,\dot{x})$.

An \define{initial submanifold} is a $\cat{Set}$-initial monomorphism in $\cat{Man}$, that is, a smooth injection $i \colon M \to N$ between manifolds such that if $L$ is a third manifold and $f \colon L \to M$ is a set-theoretic map such that $i \circ f$ is smooth, then $f$ is smooth (see~\cite{cats} for details on initial morphisms). This is also called a weak embedding, and it implies that $i$ is an immersion. This definition ensures that $M$ has a unique smooth structure making it an immersed submanifold of $N$ of given dimension.

%% file: chap2.tex
\chapter{The tangent functor monad and its algebras}
\label{chap:tfmonad}

In this chapter, we define the tangent functor monad in some categories of manifolds, namely, categories with a good notion of differential calculus and stable under the tangent functor.
We briefly recall the general definitions of a monad and its algebras, and what they entail in our case. We refer to~\cite[Chap.VI]{maclane} for more details on monads and their algebras.

We prove the uniqueness of the tangent functor monad.
Then we give some examples and general properties of algebras for this monad, as well as the general constructions of induced, quotient, and lifted algebras.
We end this chapter with a study of an analogue of the tangent functor monad in algebraic geometry.

\section{The tangent functor monad}
\label{sec:tfm}

\subsection{Monads}

We first give a general, concise definition of a monad, that we will make more explicit in our special case.

\begin{defi}[Monad]
A \define{monad} in a category is a monoid object in the endofunctor category of that category.
\end{defi}

The endofunctor we consider here is the tangent functor $T$ of $\cat{Man}$, which to a manifold associates the total space of its tangent bundle (so that we ``forget'' the bundle structure of $TM$) and to a map $f$ associates its tangent map $Tf$.

A monoidal structure on $T$ is, as in a usual monoid, given by a unit $\zeta \colon \id_{\cat{Man}} \Longrightarrow T$ and a multiplication $\mu \colon T^2 \Longrightarrow T$, which are two natural transformations satisfying the unit axioms $\mu \circ \zeta T = \mu \circ T\zeta = \id_T$ and the associativity axiom $\mu \circ \mu T = \mu \circ T\mu$, where juxtaposition denotes the horizontal composition of natural transformations. These can be pictured in the following diagrams of natural transformations:
\begin{equation*}
\xymatrix{
T^3 \ar[r]^{\mu T} \ar[d]_{T \mu} & T^2 \ar[d]^{\mu}\\
T^2 \ar[r]_{\mu} & T
}
\qquad\text{and}\qquad
\xymatrix{
T \ar[r]^{ \zeta T} \ar[dr]_{\id_T} & T^2 \ar[d]^{\mu} & T \ar[l]_{T \zeta} \ar[dl]^{\id_T}\\
& T
}
\end{equation*}

\subsection{The tangent functor monad}

A natural candidate for the unit is the \define{zero section}
\begin{equation}
\zeta \colon \id_{\cat{Man}} \Rightarrow T
\end{equation}
given by $\zeta_M \colon x \mapsto (x,0)$. Therefore we want to find a multiplication $\mu$ satisfying the unit laws, which in charts read
\begin{equation*}
\mu_M (x,v,0,0) = \mu_M (x,0,v,0) = (x,v).
\end{equation*}
Another important natural transformation is the tangent bundle \define{projection}
\begin{equation}
\tau \colon T \Rightarrow \id_{\cat{Man}}
\end{equation}
given by $\tau_M \colon (x,v) \mapsto x$.
The zero section is a retraction for the projection, that is, $\tau \circ \zeta = \id_{\id_{\cat{Man}}}$. This gives two natural candidates for the multiplication: $\tau T$ and $T\tau$, given respectively by $\tau_{TM} \colon (x,v,\dot{x},\dot{v}) \mapsto (x,v)$ and $T\tau_M \colon (x,v,\dot{x},\dot{v}) \mapsto (x,\dot{x})$ in charts for a manifold $M$. Each  satisfies one of the two unit laws. Better, their sum satisfies both. The next theorem shows that it does indeed yield a monad.

\begin{thm}\label{thm:tfm}
The triple $(T, \zeta, \mu)$ where $\zeta$ is the zero section and
\begin{equation}
\mu = \tau T + T \tau
\end{equation}
is a monad in the category of manifolds.
\end{thm}

In charts of a manifold $M$, we have
\begin{equation}
\mu_M(x,v,\dot{x},\dot{v}) = (x,v+\dot{x}).
\end{equation}

\begin{defi}[Tangent functor monad]
The \define{tangent functor monad} is the monad $(T, \zeta, \mu)$ of Theorem~\ref{thm:tfm}.
\end{defi}

We now give an ``element-free'' definition of the sum of two appropriate natural transformations. If $p \colon N \to B$ is a vector bundle and $f, g \colon M \to N$ are smooth maps equalizing $p$, one can define $f +_p g = f+g \colon M \to N$ pointwise in an obvious way. As with any pointwise operation, composition right-distributes over sum: if $h \colon L \to M$, then $(f+g) \circ h = f \circ h + g \circ h$. As for left-distributivity, if $k \colon N \to N'$ is a vector bundle map over $k_0 \colon B \to B'$, then $k \circ (f +_p g) = k \circ f +_{p'} k \circ g$.

Now, let $q, p \colon \cat{C} \to \cat{VecBun}$ be functors.
If $\alpha, \beta \colon q \Rightarrow p$ are natural transformations equalizing $p$ (for any object $c$, $p(c) \circ \alpha_c = p(c) \circ \beta_c$), then we define the sum $\alpha +_p \beta$ ``objectwise''.
It is a natural transformation because $p$ produces vector bundle maps.
Both horizontal composition and vertical composition right-distribute over sum. Some forms of left-distributivity also hold, but they are more technical and we will not need them.

Therefore, when defining $\mu =\tau T +_T T \tau$, we consider that the tangent functor $T$ has values in $\cat{VecBun}$ (and similarly for $\tau$) and it is only after summation that we ``forget'' the bundle structure.

\begin{proof}[Proof of Theorem~\ref{thm:tfm}]
This is a straightforward computation in local coordinates.
For the unit laws,
\[
(x,v) \xrightarrow{\zeta_{TM}} (x,v,0,0) \xrightarrow{\mu_M} (x,v)
\qquad\text{and}\qquad
(x,v) \xrightarrow{T\zeta_{M}} (x,0,v,0) \xrightarrow {\mu_M} (x,v)
\]
and for associativity,
\[
(x,v,\dot{x},\dot{v},x',v',\dot{x}',\dot{v}') \xrightarrow{\mu_{TM}}
(x,v,\dot{x}+x',\dot{v}+v') \xrightarrow{\mu_{M}} (x,v+\dot{x}+x')
\]
and
\[
(x,v,\dot{x},\dot{v},x',v',\dot{x}',\dot{v}') \xrightarrow{T\mu_{M}}
(x,v+\dot{x},x',\dot{v}+v') \xrightarrow{\mu_{M}} (x,v+x'+\dot{x})
\]
which are indeed equal. We can also notice that the unit laws are consequences of right-distributivity of composition over sum and of the relation
\begin{equation*}
\tau T \circ T\zeta = T \tau \circ \zeta T = \zeta \circ \tau \colon T \Rightarrow T
\end{equation*}
which is given by $(x,v) \mapsto (x,0)$ on any manifold.
A similar element-free proof of associativity can be given by carefully using left-distributivity.
\end{proof}

\section{Uniqueness of the tangent functor monad}

\begin{prop}
\label{prop:unique}
The tangent functor monad is the unique monad on the tangent functor in the smooth and analytic categories.
\end{prop}

The main reason why this result holds is that natural transformations between functors derived from the tangent functor are scarce: naturality is a strong requirement. A complete classification for a large class of such functors, at least in the smooth category, has been given by Kainz, Kol\'{a}\v{r}, Michor and Slov\'{a}k. We will discuss their results at the end of this section. First, we directly prove the results we need while making minimal assumptions on the categories of manifolds considered.

\begin{lem}
\label{lem:uniqNat}
Let $\cat{C}$ be a subcategory of $\cat{Man}$ stable under the tangent functor.
Let $m, n \geq 0$ be natural numbers and let $\alpha \colon T^m \Rightarrow T^n$ be a natural transformation between these two endofunctors of $\cat{C}$.

If $\cat{C}$ contains constant maps with value any point of any manifold in $\cat{C}$, then $\alpha \circ \zeta^m = \zeta^n$.
In particular, if $m=0$, then $\alpha = \zeta^n$.

If furthermore each point of each manifold of $\cat{C}$ is distinguished by $\cat{C}$-maps which are $m$-flat at that point,
then $\tau^n \circ \alpha = \tau^m$.
In particular, if $n=0$, then $\alpha = \tau^m$.
\end{lem}

A point $x \in X$ is \emph{distinguished} by a set $E$ of maps from $X$ if for any $y \in X \setminus \{x\}$ there is an $f \in E$ with $f(y) \neq f(x)$. This condition cannot be removed. For example, if $\cat{C}$ is the category of complete affine manifolds and affine maps (where all maps 1-flat at a point are constant), the assignment $\alpha_M(x,v) = x + \lambda v$, $\lambda \in \KK$, defines a natural transformation $\alpha \colon T \Rightarrow \id$ which is different from $\tau$ as soon as $\lambda \neq 0$. We will study the affine category in more detail in Chapter~\ref{chap:aff}.

\begin{proof}[Proof of Lemma~\ref{lem:uniqNat}]
For the first claim, if $x \in M \in |\cat{C}|$, then naturality of $\alpha$ with respect to a map constant equal to $x$ reads $\alpha_M(0^m_x) = 0^n_x$.
For the second claim, if $x \in M \in |\cat{C}|$ and $f \colon M \to N$ is $m$-flat at $x$, then naturality of $\alpha$ with respect to $f$ gives $\alpha_N(0^m_{f(x)}) = T^nf(\alpha_M(x,v))$ for any $v \in T_x^mM$, so by the first claim $0^n_{f(x)} = T^nf(\alpha_M(x,v))$, and applying $\tau^n_N$ gives $f(x) = f(\tau^n_M (\alpha_M(x,v)))$, so $x$ and $\tau^n_M (\alpha_M(x,v))$ are separated by no map which is $m$-flat at $x$, so $\tau^n_M (\alpha_M(x,v)) = x$.
\end{proof}

\begin{rmk}
If any open subset of a manifold in $\cat{C}$ is in $\cat{C}$, along with its inclusion, then naturality with respect to inclusions shows at once that $\alpha_M$ sends fibers over $x$ to fibers over $x$ (which are not vector spaces in general), by using the filter of open neighborhoods of $x$. This can be used in the previous lemma to weaken the hypothesis: it is then enough that each point be told apart by locally defined $m$-flat maps. Of course, all of this is not needed in the smooth category, where we have partitions of unity, but could prove useful in more rigid categories.
Note also that in the proof, we have not used continuity of the functions $\alpha_M$.
\end{rmk}

\begin{lem}
\label{lem:T2T}
A natural transformation from $T^2$ to $T$ in the smooth or analytic category is of the form $a \tau T + b T \tau$ with $a, b \in \KK$.
\end{lem}

In the statement of the lemma, the multiplications by scalars of natural transformations are defined similarly to the sum we defined in the previous section.

\begin{proof}
Let $\alpha \colon T^2 \Rightarrow T$ be a natural transformation and let $\xi \in T^2M$ with $M \in |\cat{C}|$. When needed, we will work in charts and write $\xi = (x,v,\dot{x},\dot{v})$. In particular, $\tau_{TM}(\xi) = (x,v)$ and $T \tau_M(\xi) = (x,\dot{x})$.

Let $f \colon M \to M$ be a $\cat{C}$-morphism. Then naturality of $\alpha$ reads $Tf(\alpha_M(\xi)) = \alpha_M(T^2f(\xi))$. Therefore, if $T^2f(\xi) = \xi$, then $Tf(\alpha_M(\xi)) = \alpha_M(\xi)$. In other words, if $\xi$ is fixed by $T^2f$, then $\alpha_M(\xi)$ is fixed by $Tf$. This implies that $\alpha_M(\xi) \in \spann (v, \dot{x}) \subseteq T_xM$. Indeed, if neither of $v$ and $\dot{x}$ is zero and $\alpha_M(\xi) \notin \spann (v, \dot{x})$, then we can construct a function $f$ such that $f'(x)$ is the identity on $\spann(v,\dot{x})$ but is zero on $\spann(\alpha_M(\xi))$, and $f''(x)(v,\dot{x}) = \dot{v} - f'(x) \dot{v}$ (since $f''(x)$ may be any continuous symmetric bilinear map). Then
\[
T^2f(x,v,\dot{x},\dot{v}) = \big( f(x),f'(x)v, f'(x) \dot{x}, f'(x)\dot{v} + f''(x)(v,\dot{x}) \big) = (x,v,\dot{x},\dot{v})
\]
so $T^2f$ fixes $\xi$, so $Tf$ fixes $\alpha_M(\xi)$, so  $\alpha_M(\xi)=0$, which is absurd. Therefore, if none of $v$ and $\dot{x}$ is zero, then $\alpha_M(\xi) \in \spann (v, \dot{x})$, and by continuity of $\alpha_M$, this is actually true also if $v$ or $\dot{x}$ is zero. Therefore there are functions $h_M$ and $k_M$ such that $\alpha_M(\xi) = h_M(\xi) v + k_M(\xi) \dot{x}$.

Now if $f \colon L \to M$ and $\xi_0 \in T^2L$, naturality reads
\[
\big( h_M(T^2f(\xi_0)) - h_L(\xi_0) \big) f'(x_0) v_0 + \big( k_M(T^2f(\xi_0)) - k_L(\xi_0) \big) f'(x_0)\dot{x}_0 =0.
\]

Let $L = \KK^2$ and $\xi_0 = (0,1,1,0)$. There is an $f$ such that $T^2f(\xi_0) = \xi \in T^2M$. Therefore if $f$ is an immersion at 0, then $h_M(\xi) = h_{\KK^2}(\xi_0) = a$ and similarly $k_M(\xi) = k_{\KK^2}(\xi_0) = b$. Therefore by continuity of $h_M$ and $k_M$, if $M$ has dimension at least 2, then $h_M$ and $k_M$ are constant equal to $a$ and $b$ respectively. If $f$ has rank at least 1, then  
$h_M(\xi) + k_M(\xi) = a+b$, so by continuity, if $M$ has dimension at least 1, then $h_M + k_M = a+b$. So in any case (including dimension 0), $\alpha$ has the required form.
\end{proof}

\begin{proof}[Proof of Proposition~\ref{prop:unique}]
If $(T,\alpha,\beta)$ is a monad on the tangent functor, then by the previous lemmas, $\alpha = \zeta$ and $\beta = a \tau T + b T \tau$ for some $a, b \in \KK$. Then the two unit laws impose $a=b=1$.
\end{proof}

There is a dual notion of comonad, namely a comonoid object in the endofunctor category of a category. More explicitly, a comonad is a triple $(T,\epsilon, \delta)$ where the counit $\epsilon \colon T \Rightarrow \id_{\cat{Man}}$ and the comultiplication $\delta \colon T \Rightarrow T^2$ satisfy the counit laws $\epsilon T \circ \delta = T \epsilon \circ \delta = \id_T$ and the coassociativity axiom $\delta T \circ \delta = T \delta \circ \delta$.
In the same vein as the above uniqueness result, we have the following.

\begin{prop}
There is no comonad on the tangent functor in the smooth and analytic categories.
\end{prop}

\begin{proof}
From the previous lemma, the counit $\epsilon$ has to be the projection $\tau$. Then the counit laws mean that $\delta$ should be a natural second-order vector field, that is, of the form $\delta_M(x,v) = (x,v,v,\alpha_M(x,v))$. If $(x,v) \in TM$, let $f \colon M \to M$ be such that $f(x)=x$ and $T_xf = \id_{T_xM}$. Then naturality of $\delta$ in a chart reads $f''(x)(v,v)=0$. But if $v \neq 0$, then in the smooth and analytic categories, there are functions $f$ with $f(x)=x$ and $T_xf = \id_{T_xM}$ but $f''(x)(v,v) \neq 0$.
\end{proof}

\begin{rmk}
The argument of the proof fails in the affine category, where there actually are comonads on the tangent functor, as we will see in Chapter~\ref{chap:aff}.
\end{rmk}

Finally, the cotangent functor is contravariant, and there is no notion of (co)monad on a contravariant ``endofunctor''. To have a covariant endofunctor, we should restrict to the groupoid of isomorphisms of the category we consider and consider the contragredient map $(T^*f)^{-1}$, but then, the unit, for instance, should produce isomorphisms, and there is no hope for this since $M$ and $T^*M$ have different dimensions (if positive).

\subsection{Weil algebras and Weil functors}

The study of large classes of bundle functors and their natural transformations has been carried out by Kainz and Michor~\cite{km}, Eck~\cite{eck} and Luciano~\cite{luci}, extending ideas of Weil~\cite{weil}.
Here we recall the main results of this theory, following the comprehensive exposition in~\cite{kms}, and see how they relate to monads in differential geometry.
In this subsection, we restrict ourselves to the category of finite dimensional smooth paracompact Hausdorff real manifolds.

A \define{bundle functor} is a functor $F \colon \cat{Man} \to \cat{FibMan}$ which is a section of the ``base'' functor $\cat{FibMan} \to \cat{Man}$ (therefore we have a natural transformation $\pi \colon F \Rightarrow \id_{\cat{Man}}$) and which is local: if $i \colon U \to M$ is an inclusion, then $FU$, which is $\pi_U^{-1}(U)$, is equal to $\pi_M^{-1}(U)$, and $Fi \colon FU \to FM$ is the inclusion.

It turns out that the product preserving bundle functors can be classified by algebraic objects, known as Weil algebras.
A \define{Weil algebra} is a unital commutative finite dimensional $\RR$-algebra of the form $A = \RR 1 \oplus N$ where $N$ is a nilpotent ideal.

To each Weil algebra $A$ is associated an endofunctor $\Phi(A) = T_A$ of $\cat{Man}$, called the \define{Weil functor} associated to $A$.
This functor preserves products and there is a projection $\pi_A \colon T_A \Rightarrow \id_{\cat{Man}}$ such that $(T_A,\pi_A)$ is a bundle functor.
Conversely, all product preserving bundle functors ``are'' Weil functors.

Here is an outline of the construction. Let $A = \RR 1 \oplus N$.
It is enough to define $T_A$ on open subsets of $\RR^m$ and on maps in $C^\infty(\RR^m,\RR^k)$ in a compatible way.
If $f \in C^\infty(\RR,\RR)$, then define $T_Af \colon A \to A$  by $T_Af(\lambda 1+n) = \sum_{i \in \NN} \frac{f^{(i)}(\lambda) }{i!} n^i$ if $\lambda 1 + n \in A$.
This is a finite sum since $N$ is nilpotent.
We use the generalization of this formula to several variables, and use it componentwise, to define $T_Af$ for any $f \in C^\infty(\RR^m,\RR^k)$. If $U \subseteq \RR^m$, then simply define $T_AU = U \times N^m$.
Since all manifolds can be constructed by gluing open subsets of $\RR^m$ via diffeomorphisms of $\RR^m$, this is enough to define $T_A$ on objects.
Similarly, $T_A$ is defined on morphisms via their representations in charts.
Then one can show that
\[
T_A = \Hom( C^\infty(-,\RR),A).
\]

The two simplest examples are the following.
\begin{itemize}
\item
If $A = \RR$ then $T_A = \id_{\cat{Man}}$. This is because any multiplicative form on $C^\infty(M,\RR)$ is the evaluation at a point of $M$. The converse is obvious.
\item
If $A$ is the algebra of real dual numbers, that is, $A$ is generated by 1 and $\epsilon$ with $\epsilon^2=0$, then $T_A = T$ is the tangent functor. Indeed, an algebra morphism $\phi \colon C^\infty(M,\RR) \to A$ induces $\phi_0 \colon C^\infty(M,\RR) \to A \twoheadrightarrow \RR$, so $\phi_0$ is the evaluation at a point $x \in M$ by the previous example. Now, $\phi(f - f(x)1) = \phi(f) - f(x)1 = \tilde{\phi}(f) \epsilon$. That $\phi$ is an algebra morphism implies that $\tilde{\phi}$ satisfies the Leibniz rule, so can be identified with a vector in $T_xM$. Again, the converse is obvious.
\end{itemize}

The above construction also shows how a Weil algebra morphism $f \colon A \to B$ yields a natural transformation $\Phi(f) \colon T_A \Rightarrow T_B$.
Then $\Phi$ is a functor from the category $\cat{WeilAlg}$ of Weil algebras and (unital) algebra morphisms to the category $\cat{WeilFunc}$ of Weil functors and natural transformations (\textit{ibid.}, Theorem~35.13).

Both $\cat{WeilAlg}$ and $\cat{WeilFunc}$ are monoidal, with monoidal products the tensor product and horizontal composition respectively. Then the functor $\Phi$ is monoidal ($T_{A \otimes B} = T_A \circ T_B$, \textit{ibid.}, Theorem~35.18) and the main result is the following.

\begin{thm}[Kainz, Kol\'{a}\v{r}, Michor, Slov\'{a}k]
The functor $\Phi \colon \cat{WeilAlg} \to \cat{WeilFunc}$ is a monoidal equivalence of categories.
\end{thm}

In particular, one can recover the Weil algebra of a Weil functor by $T_A(\RR)=A$.

\begin{proof}
This is a reformulation of Theorem~35.17 and Proposition 35.18 p.307 in~\cite{kms}.
\end{proof}

This makes natural transformations of product preserving bundle functors easier to classify, by first classifying Weil algebra morphisms.

To see how this relates to monads in $\cat{Man}$, let $A$ be a Weil algebra. Then the inclusion $\RR \hookrightarrow A$ induces a natural transformation $\zeta_A \colon \id_{\cat{Man}} \Rightarrow T_A$, and the multiplication $m \colon A \otimes A \to A$ gives a natural transformation $\Phi(m) = \mu_A \colon T_A^2 \Rightarrow T_A$. The unit laws and associativity of $A$ imply that $(T_A,\zeta_A, \mu_A)$ is a monad in $\cat{Man}$. The tangent functor monad is the ``simplest'' example if this type, after the identity monad.

We also obtain at once the uniqueness of natural transformations $\id_{\cat{Man}} \Rightarrow T_A$. The Weil algebra inducing $T^2$ is the tensorial square of the algebra of dual numbers, so with obvious notations, natural transformations $T^2 \Rightarrow T$ correspond to algebra morphisms
\[
f \colon \spann (1,\epsilon_1,\epsilon_2, \epsilon_1 \epsilon_2) \to \spann(1,\epsilon)
\]
and we recover Lemma~\ref{lem:T2T} for real smooth manifolds: $f$ has to map 1 to 1, $\epsilon_i$ to $a_i \epsilon$, and $\epsilon_1 \epsilon_2$ to 0.

\section{Algebras over the tangent functor monad}
\label{sec:alg}

The idea of an algebra, or module, over a monad is similar the one of group representation, or group action, or in the present case: monoid action.

\begin{defi}[Algebra over a monad]
An \define{algebra} over the monad $(T,\eta,\mu)$ in the category $\cat{C}$ and on the object $M \in |\cat{C}|$ is a $\cat{C}$-morphism $h \colon TM \to M$ such that $h \circ \eta_M = \id_M$ and $h \circ Th = h \circ \mu_M$.
\end{defi}

The latter two equations say respectively that acting by the unit is like not acting at all and acting successively by two elements is like acting by their product.

Therefore, an algebra for the tangent functor monad and on a manifold $M$ is a map
\begin{equation*}
h \colon TM \to M
\end{equation*}
such that $h \circ \zeta_M = \id_M$, that is, $h$ is a retraction of the zero section, and $h \circ Th = h \circ \mu_M$. In charts, the two axioms read
\begin{align}
h(x,0) &= x\\
h \big( h(x,v), h'(x,v)(\dot{x},\dot{v}) \big) &= h(x, v+\dot{x})\label{eq:algaxb}
\end{align}
for $x \in M$, $v \in T_xM$ and $(\dot{x},\dot{v}) \in T_{(x,v)}TM$. We shall sometimes call $M$ the \define{base manifold} of $h$.

\begin{rmk}
An algebra for the identity monad is a map $h \colon M \to M$ such that $h \circ h = h$.
\end{rmk}

\subsection{Two basic examples}

The easiest example of an algebra on $M$ is the tangent bundle projection $\tau_M \colon TM \to M$, which deserves the name of \define{trivial algebra}. Another example of algebra is the \define{free algebra} on $TM$ given by the multiplication $\mu_M \colon T^2M \to TM$ itself. That the multiplication of a monad produces algebras for this monad is a general phenomenon (see~\cite{maclane}).
We will see below that these two families are in some precise sense the two extreme examples of algebras.

\subsection{Algebra morphisms and the category of algebras}

\begin{defi}[Morphism of algebras]
Given a monad on the endofunctor $T$ in the category $\cat{C}$, a \define{morphism} from the algebra $h$ on $M$ to the algebra $k$ on $N$ is a $\cat{C}$-morphism $f \colon M \to N$ such that the diagram
\begin{equation*}
\xymatrix{
TM \ar[d]_{Tf} \ar[r]^h & M \ar[d]^f\\
TN \ar[r]_k & N
}
\end{equation*}
commutes.
\end{defi}

Algebras and their morphisms form a category $\cat{C}^T$ concrete over $\cat{C}$ (the Eilenberg--Moore category of $T$). There is a natural adjunction between $\cat{C}$ and $\cat{C}^T$; see~\cite{maclane} for more details on the relations between monads and adjunctions.

As for the two basic examples: A morphism from $\tau_M$ to $\tau_N$ is any map from $M$ to $N$ (there is no additional constraint). A morphism from $\mu_M$ to $\mu_N$ is a map $f \colon TM \to TN$ such that in charts
\[
f(x,v+\dot{x}) = \big( \tau_N(f(x,v)), f(x,v) + \pr_1( f'(x,v)(\dot{x},\dot{v})) \big)
\]
so in particular, $f$ is a fiber bundle map over $f_0 \colon M \to N$.

The equality in the fiber direction with $v=0$ implies $f(x,\dot{x}) = f(x,0) + {f_0}'(x)\dot{x}$. In other words, $f = Tf_0 + (f \circ \zeta_M) \circ \tau_M$, that is, $f$ is an affine tangent map.
Conversely, all affine tangent maps are algebra morphisms between the associated free algebras.

\subsection{Products of algebras}

The category $\cat{Man}$ is monoidal for the cartesian product, and the tangent functor monad is monoidal for this structure. This means that $T$ is a monoidal functor and $\zeta$ and $\mu$ are monoidal natural transformations (all of this is straightforward to check). This implies that the category of algebras over the tangent functor monad is monoidal. Explicitly, if $h$ is an algebra on $M$ and $k$ is an algebra on $N$, then $h \times k$ is an algebra on $M \times N$, and similarly for algebra morphisms.

On the other hand, the tangent map of an algebra need not be an algebra. Actually, the tangent map of any map $f \colon TM \to M$ with $M$ of positive dimension fails to satisfy the unit axiom: in charts, $Tf(x,v,0,0) = (f(x,v),0) \neq (x,v)$ as soon as $v \neq 0$.

\section{General properties of algebras}

Since $h(x,0)=x$, we have the following easy but important result.

\begin{prop}
An algebra is a submersion on a neighborhood of the zero section.
\end{prop}

\begin{proof}
Differentiating $h(x,0) = x$, we obtain $h'(x,0) = \pr_1$, which is surjective. Therefore, $h$ is a submersion on the zero section, and submersivity is an open condition.
\end{proof}

Several important properties of algebras can be derived from the second axiom~\eqref{eq:algaxb} by differentiation and setting the variables to special values.
We define the following maps and spaces. If $h \colon TM \longrightarrow N$ is a map and $x \in M$, we define the map
\begin{equation}
h_x \colon T_xM \longrightarrow N
\end{equation}
by $h_x(v) = h(x,v)$ if $v \in T_xM$. If furthermore $N=M$ and $h(x,0)=x$ for all $x \in M$, we define the endomorphism
\begin{equation}
A_x = {h_x}'(0) \in \calL(T_xM)
\end{equation}
and the linear subspace
\begin{equation}
\calD_x = \im A_x \subseteq T_xM.
\end{equation}

We call $A \in \Gamma(\calL(TM))$ (or $A^h$ when needed) the \define{endomorphism associated to $h$}, and $\calD \subseteq TM$ (or $\calD^h$ when needed) the \define{distribution induced by $h$} on the base manifold $M$.

As for the two basic examples: For the tangent bundle projection $\tau_M$, one has $A_x = 0$, so $\calD_x = \{0\}$ and $\ker A_x = T_xM$.

In the case of the free algebra $h=\mu_M$, which is an algebra on $TM$, one has $h_{(x,v)} \colon (\dot{x},\dot{v}) \mapsto (x,v+\dot{x})$ for any $(x,v) \in TM$, so $A_{(x,v)} \colon (\dot{x},\dot{v}) \mapsto (0,\dot{x})$, so $A$ is the canonical endomorphism of $T^2M$. Therefore, $\calD_{(x,v)} = \im A_{(x,v)} = \ker A_{(x,v)} = T_{(x,v)}^{\mathrm{vert}}TM$ is the space of vertical vectors.

For a product of algebras, it is easy to check that $A^{h \times k}_{(x,y)} = A^h_x \times A^k_y$ as endomorphisms of $T_{(x,y)}(M \times N) \simeq T_xM \times T_yN$, and $\calD^{h \times k}_{(x,y)} = \calD^h_x \times \calD^k_y$.

With the notations above, differentiating the first axiom for algebras, we get $h'(x,0)(\dot{x},\dot{v}) = \dot{x} + A_x\dot{v}$.
Using this in the second axiom with $v=0$, we obtain $h(x, \dot{x} + A_x \dot{v} ) = h(x,\dot{x})$,
therefore
\begin{equation}\label{eq:D-invar}
h(x, \dot{x} + \calD_x) = h(x,\dot{x})
\end{equation}
hence
\begin{equation}\label{eq:incl-ker}
\calD_x \subseteq \ker {h_x}'(\dot{x})
\end{equation}
for all $\dot{x} \in T_xM$. In particular, for $\dot{x}=0$, this reads $\im A_x \subseteq \ker A_x$, that is,
\begin{equation}\label{eq:nilp}
{A_x}^2 = 0.
\end{equation}

We will see in Chapter~\ref{chap:aff} on affine algebras that, conversely, any linear endomorphism squaring to 0 can appear as the associated endomorphism of an algebra at some point.

Now we can see in which sense the two families of examples we gave are two extreme cases: for the tangent bundle projections, $A_x=0$, while for the free algebras, $\im A_x = \ker A_x$.

We define the \define{rank} of an algebra at $x$ to be the rank, finite or infinite, of the associated endomorphism $A_x$. It is a lower semicontinuous function of $x$. The rank of an algebra is defined to be the maximum of the ranks at all points.
If $M$ has finite dimension, then the rank of an algebra on $M$ is at most half the dimension of $M$, since ${A_x}^2=0$. All these values occur: for instance, the algebra
\[
\tau_{\KK^m} \times \mu_{\KK^n}
\]
on $\KK^m \times (T\KK)^n \simeq \KK^{m+2n}$ has rank $n$. Other examples arise from affine algebras, studied in the next chapter.

We say that an algebra is \define{regular} if $\ker A$ and $\im A$ are subbundles.
For algebras on Hilbert manifolds, it is sufficient that $\im A$ be a subbundle, which is the case as soon as $A$ has locally constant finite rank.

By analogy with the finite-dimensional case, we say that an algebra has \define{maximal rank} if it is regular and $\im A = \ker A$.
On Hilbert manifolds, it is sufficient that $\im A = \ker A$.

Finally, we say that an algebra on $M$ is \define{tame at $x \in M$} if $h_x$ is a subimmersion at the origin, which is the case if it has locally constant finite rank. An algebra is called \define{tame} if it is tame at all points.

We have derived an important property of the kernels of the ${h_x}'(\dot{x})$, namely that $\calD_x$ is included in all of them. Now we are going to derive a similar property of the images of the ${h_x}'(\dot{x})$, namely that $\calD$ contains all of them.
Differentiating the second axiom for algebras with respect to $(\dot{x},\dot{v})$ gives
\[
{h_{h(x,v)}}' \big( h'(x,v) (\dot{x},\dot{v}) \big) \circ h'(x,v) = {h_x}'(v+\dot{x}) \circ \pr_1
\]
so if $(\dot{x},\dot{v})= (0,0)$, then
\[
A_{h(x,v)} \circ h'(x,v) = {h_x}'(v) \circ \pr_1
\]
and since $\pr_1$ is surjective,
\begin{equation}\label{eq:incl-im}
\im {h_x}'(v) \subseteq \calD_{h(x,v)}
\end{equation}
with equality for small $v$, for instance for $(x,v)$ in the neighborhood of the zero section where $h$ is a submersion.
This has the following direct consequence.

\begin{prop}
A regular algebra is tame.
\end{prop}

As an easy application, we classify the algebras of rank zero, that is, the algebras with associated endomorphism the zero map, and their morphisms.

\begin{prop}[Algebras of rank 0 and their morphisms]
The algebras of rank 0, and in particular the algebras on manifolds of dimension at most 1, are the tangent bundle projections.
The morphisms between algebras of rank 0 are all the maps between the corresponding base manifolds.
\end{prop}

\begin{proof}
From the inclusion~\eqref{eq:incl-im}, for any $(x,v) \in TM$ we have $\im {h_x}'(v) \subseteq \calD_{h(x,v)} = \{0\}$, so all the $h_x$ are constant, and necessarily equal to $h(x,0)=x$.
By the above, the rank of an algebra on a manifold of dimension at most 1 is at most $\frac12$, so it is 0.
The statement concerning morphisms is obvious.
\end{proof}

We summarize the inclusions~\eqref{eq:incl-ker} and~\eqref{eq:incl-im} informally by saying that $\calD_x$ is included in all the kernels and contains all the images.

We can actually derive a slight strengthening of~\eqref{eq:incl-im}. Differentiating the second axiom for algebras with respect to $(v,\dot{x},\dot{v})$ at $v=\dot{v}=0$ gives
\[
h'(x,\dot{x}) \big( A_x v', \dot{x}' + A_x \dot{v}' +  h''(x,0)((0,v'),(\dot{x},0)) \big) = {h_x}' (\dot{x}) (\dot{x}'+v').
\]
We write symbolically $(\calD_x, \ast) = {{\tau_M}'(x,v)}^{-1} \calD_x$.
Since $v'$ and $\dot{x}'$ are arbitrary in the above equation, we obtain
\begin{equation}
h'(x,\dot{x}) (\calD_x, \ast) = \im {h_x}'(\dot{x}) \subseteq \calD_{h(x,\dot{x})}
\end{equation}
where the last inclusion is~\eqref{eq:incl-im}. We will see in Chapter~\ref{chap:foli} that, for tame algebras, this is an immediate consequence of the interpretation of the $\calD_x$'s as tangent spaces of leaves of a foliation.

\begin{rmk}(Restriction to connected manifolds)
If $h$ is an algebra on $M$ and $(x,v) \in TM$, then $h(x,v)$ is in the connected component of $x$. Indeed, $x$ and $h(x,v)$ are the endpoints of the path $[0,1] \to M, t \mapsto h(x,tv)$. Therefore, we can restrict our study to the category of connected manifolds without loss of generality.
\end{rmk}

\subsection{The intertwining property of algebra morphisms}

We have the following easy fact concerning algebra morphisms.

\begin{prop}\label{prop:intertw}
The tangent map of a morphism between algebras intertwines their associated endomorphisms, and in particular preserves the induced distributions.
More precisely, if $f \colon M \to N$ is a morphism from the algebra $h$ on $M$ with associated endomorphism $A$ to the algebra $k$ on $N$ with associated endomorphism $B$, then the diagram
\begin{equation*}
\xymatrix{
TM \ar[d]_{Tf} \ar[r]^A & TM \ar[d]^{Tf}\\
TN \ar[r]_B & TN
}
\end{equation*}
commutes, and in particular $Tf (\calD^h) \subseteq \calD^k$ and $Tf (\ker A) \subseteq \ker B$ and $B \circ Tf \circ A = 0$.
\end{prop}

\begin{proof}
Differentiating the morphism axiom $k(f(x),f'(x)v)=f(h(x,v))$ with respect to $v$ at $v=0$ gives the result.
\end{proof}

If at least one of the source and the target algebras is a tangent bundle projection, then the intertwining condition is also sufficient. Namely, the morphisms from $\tau_M$ are the maps $f$ with $B \circ Tf = 0$, and the morphisms to $\tau_N$ are the maps $f$ with $Tf \circ A = 0$ (as we will see in Chapter~\ref{chap:foli}, this means that $f$ has to be constant on the leaves of $h$, which are the sets $h(x,\ast)$).

\section{Subalgebras, quotient algebras, and lifted algebras}

In this section, we construct new algebras from old ones. The third subsection about covering algebras will be used in classifying affine algebras in Chapter~\ref{chap:aff} and to reduce the study of algebras to oriented ones (to be defined) by lifting any algebra to its oriented cover.

\subsection{Subalgebras}

If $h$ is an algebra on a manifold $M$ and $i \colon N \to M$ is a smooth injection, then there is at most one (set-theoretic) map $h^i \colon TN \to N$ making the diagram
\begin{equation*}
\xymatrix{
TN \ar@{^(->}[d]_{Ti} \ar[r]^{h^i} & N \ar@{^(->}[d]^i\\
TM \ar[r]_h & M
}
\end{equation*}
commute, and there is such a map if and only if $h(\im Ti) \subseteq \im i$. In that case, if $i$ is a weak embedding (that is, a $\cat{Set}$-initial immersion), then $h^i$ is smooth. It is an algebra: indeed,
\[
i \circ h^i \circ \zeta_N = h \circ Ti \circ \zeta_N = h \circ \zeta_M \circ i = i
\]
so $h^i \circ \zeta_N = \id_N$, and
\[
i \circ h^i \circ Th^i = h \circ Ti \circ Th^i = h \circ Th \circ T^2i = h \circ \mu_M \circ T^2i =  h \circ Ti \circ \mu_N = i \circ h^i \circ \mu_N
\]
so $h^i \circ Th^i = h^i \circ \mu_N$.
This algebra is said to be \define{induced} by $h$ on $N$ (when $i$ is implied, for instance if it is the inclusion of a weakly embedded submanifold), or to be a \define{subalgebra} of $h$.

An example is the inclusion of a union of leaves of $h$ (sets $h(x,\ast)$). If these leaves form a totally path-disconnected set in the leaf space, then the induced algebra is the tangent bundle projection: if $(x,v) \in TN$, then $v \in T_x(h(x,\ast))$ which we will see in Chapter~\ref{chap:foli} makes sense (the leaves are weakly embedded submanifolds) and is equal to $\calD_x$, which is included in $\ker A_x$, therefore $h^i(x,v)=h(x,v)=x$.

\subsection{Quotient algebras}

An equivalence relation $\Gamma$ on a manifold $M$ is called \define{regular} if there is a manifold structure on the quotient set $M/\Gamma$ making the quotient (set-theoretic) map $p \colon M \twoheadrightarrow M/\Gamma$ a submersion.
Such a manifold structure is then unique, and it induces the quotient topology.
Recall Godement's criterion (see~\cite{serre}) that an equivalence relation on a manifold $M$ is regular if and only if its graph $\Gamma$ is an embedded submanifold of $M \times M$ and $\pr_1 \colon \Gamma \to M$ (equivalently $\pr_2$) is a submersion.
Note that this gives the most general example of quotient manifold, since any submersion $p \colon M \to N$ induces the regular equivalence relation ``being in the same fiber'', for which $N$ is the quotient manifold.
From now on, we identify an equivalence relation with its graph, and $T(M \times M)$ with $TM \times TM$.

The following fact is probably classical (it is for instance a special case of more general theorems on quotient Lie groupoids), but we reprove it here for completeness.

\begin{lem}
If $\Gamma$ is a regular equivalence relation on $M$, then $T\Gamma$ is a regular equivalence relation on $TM$, and the map
\[
\widetilde{Tp} \colon TM/T\Gamma \to T(M/\Gamma)
\]
induced by $Tp \colon TM \to T(M/\Gamma)$ is a diffeomorphism.
\end{lem}

\begin{proof}
The tangent map of an embedding is an embedding, so $T\Gamma$ is an embedded submanifold of $TM \times TM$. The projection $T\pr_1 \colon T\Gamma \to TM$ is the tangent map of a submersion so is also a submersion. The relation $T\Gamma$ is easily seen to be reflexive and symmetric. As for transitivity, suppose $(x,u) \sim (y,v) \sim (z,w)$. Since $T\pr_1$ is a submersion, it has a local section $s_1$ at $(y,v)$ such that $Ts_1(y,v)=(z,w)$ and similarly there is a section $s_2$ of $T\pr_2$ at $(y,v)$ such that $Ts_2(y,v)=(x,u)$. If $y_t$ is a curve representing $(y,v)$, then $s_1(y_t) \sim y_t \sim s_2(y_t)$ for small $t$, so by transitivity of $\Gamma$, $s_1(y_t) \sim s_2(y_t)$, but these curves represent respectively $(z,w)$ and $(x,u)$, which are therefore equivalent via $T\Gamma$.
So by Godement's criterion, $T\Gamma$ is a regular equivalence relation on $TM$.

For the second claim, the map $Tp  \colon TM \to T(M/\Gamma)$ factors through $p_T \colon TM \to TM/T\Gamma$ (same argument as in the previous paragraph) and the map $\widetilde{Tp} \colon TM/T\Gamma \to T(M/\Gamma)$ is therefore a submersion, since $Tp$ is. It is also bijective, so it is a diffeomorphism.
\end{proof}

Let $h$ be an algebra on a manifold $M$ and $\Gamma$ be a regular equivalence relation on $M$. Since $Tp$ is surjective, there is at most one (set-theoretic) map $h^\Gamma \colon T(M/\Gamma) \to M/\Gamma$ making the diagram
\begin{equation*}
\xymatrix{
TM \ar[r]^{h} \ar@{>>}[d]_{Tp} & M \ar@{>>}[d]^p\\
T(M/\Gamma) \ar[r]_{h^\Gamma} & M/\Gamma
}
\end{equation*}
commute and by the above there is one if and only if
\begin{equation}
(h \times h) (T\Gamma) \subseteq \Gamma.
\end{equation}
It is denoted by $h^\Gamma$ and given by
\begin{equation}
h^\Gamma \circ Tp = p \circ h.
\end{equation}
In that case, it is smooth since $h^\Gamma \circ Tp$ is smooth and $Tp$ is a submersion (hence a $\cat{Set}$-final morphism). It is an algebra: for the first axiom,
\[
h^\Gamma \circ \zeta_{M/\Gamma} \circ p = h^\Gamma \circ Tp \circ \zeta_M = p \circ h \circ \zeta_M = p
\]
so $h^\Gamma \circ \zeta_{M/\Gamma} = \id_{M/\Gamma}$. For the second axiom,
\[
h^\Gamma \circ Th^\Gamma \circ T^2p = h^\Gamma \circ Tp \circ Th = p \circ h \circ Th = p \circ h \circ \mu_M = h^\Gamma \circ Tp \circ \mu_M = h^\Gamma \circ \mu_{M/\Gamma} \circ T^2p
\]
so, $T^2p$ being surjective, $h^\Gamma \circ Th^\Gamma = h^\Gamma \circ \mu_{M/\Gamma}$. This algebra is called the \define{quotient algebra} of $h$ by $\Gamma$.

A sufficient condition for existence of $h^\Gamma$ is that $\Gamma$ be saturated with respect to the foliation induced by $h$ (see below). Indeed, in this case $h(x,v) \sim x$ and $h(y,w) \sim y$, so $(x,v) \sim (y,w)$ implies $x \sim y$, which implies $h(x,v) \sim h(y,w)$. This shows that the quotient algebra on $M/\Gamma$ is then the tangent bundle projection.
If the equivalence relation is given by a group action, this condition is that $G$-orbits be unions of leaves of $h$.

\paragraph{Group action}
We are going to give a sufficient condition for existence of a quotient algebra when the equivalence relation is given by a group action. Let $G$ be a Lie group acting smoothly on a manifold $M$. It determines an equivalence relation with graph $\Gamma = \{ (x,g \cdot x) \mid x \in M \text{ and } g \in G \}$. Suppose that this equivalence relation is regular. (A useful sufficient condition is that $G$ act freely and properly.) Then, if $h$ is an algebra on $M$, there exists a quotient algebra as soon as $G$ acts by automorphisms of $h$ and the action is tangential, meaning that
\[
T_x(G\cdot x) \subseteq \calD_x
\]
for all $x \in M$.

Indeed, one has $T_{[x]}(M/G) \simeq T_xM/\frakg_x$ for all $x \in M$, where we wrote $\frakg_x = T_x(G \cdot x)$. Therefore elements of $T(M/G)$ can be written $[(x, v+\frakg_x)]$ where $v \in T_xM$, and $[(x, v+\frakg_x)] = [(y, w+\frakg_y)]$ if and only if there is a $g \in G$ such that $(y,w+\frakg_y) = g \cdot (x,v + \frakg_x)$. 

Since by hypothesis $\frakg_x \subseteq \calD_x$ and $\frakg_y \subseteq \calD_y$, the terms $\frakg_x$ and $\frakg_y$ in the equality are transparent when images by $h$ are considered, and since $G$ acts by automorphisms, $(y,w) = g \cdot (x,v)$ implies $h(y,w) = g \cdot h(x,v)$. Therefore we can define
\[
h^\Gamma([x, v+\frakg_x]) = [h(x,v)]
\]
for $(x,v) \in TM$, and this is the quotient algebra.

\begin{rmk}
Acting by automorphisms of $h$ does not imply being tangential ($\frakg_x \subseteq \calD_x$), as shows the example of $\mu_\RR$ with $G=\RR$ acting by translations on the base $\RR$. In the plane $T\RR \simeq \RR^2$, one has $\calD_{(x,y)} = \langle \partial_y \rangle$ while $\frakg_{(x,y)} = \langle \partial_x \rangle$.
\end{rmk}

\subsection{Lifted algebras on covering spaces}

Let $p \colon M^\sharp \to M$ be a covering map. In particular, it is \'{e}tale, so we will identify the tangent spaces $T_x M^\sharp$ and $T_{p(x)} M$ via $T_x p$.

\begin{prop}[Lifted algebra]
Let $h$ be an algebra on $M$ and $p \colon M^\sharp \to M$ be a covering.
Then there is a unique algebra $h^\sharp$ on $M^\sharp$ such that $p$ is a morphism from $h^\sharp$ to $h$, that is, such that the diagram
\begin{equation*}
\xymatrix{
TM^\sharp \ar[r]^{h^\sharp} \ar[d]_{Tp} & M^\sharp \ar[d]^p\\
TM \ar[r]_h & M
}
\end{equation*}
commutes. It is defined by $h^\sharp(x,v) = \tilde{\gamma}_{x,v}(1)$ where $\tilde{\gamma}_{x,v}$ is the lift of the path $\gamma_{p(x),v} \colon t \mapsto h(p(x),tv)$ starting at $x$, and it is called the \define{lifted algebra} of $h$ by $p$.
The deck transformations of $p$ are automorphisms of $h^\sharp$.
\end{prop}

\begin{proof}
For any $(x,v) \in TM^\sharp$ and $t \in \RR$, we need $p(h^\sharp(x,tv)) = h(p(x),tv)$. By the first axiom of algebras, we also need $h^\sharp(x,0)=x$. Uniqueness follows from the uniqueness of the lift of a path with given starting point.
Conversely, the map $h^\sharp$ thus defined makes the above diagram commute and satisfies the first axiom for algebras. For the second axiom, with the identification of tangent spaces,
${h^\sharp}'(x,v)(\dot{x},\dot{v}) = \partial_{(x,v)}\tilde{\gamma}_{x,v}(1) (\dot{x},\dot{v}) = \partial_{(x,v)}\gamma_{p(x),v}(1) (\dot{x},\dot{v}) = h'(p(x),v)(\dot{x},\dot{v})$, vector that we denote by $w$.
So we have to check that $\tilde{\gamma}_{\tilde{\gamma}_{x,v}(1),w} (1) = \tilde{\gamma}_{x,v+\dot{x}}(1)$. 
The path on the left hand side is the lift of the path $\gamma_{h(p(x),v),w}$ starting at $\tilde{\gamma}_{x,v}(1)$,
and the path on the right hand side is the lift of the path $\gamma_{p(x),v+\dot{x}}$ starting at $x$.
Since the lifts starting at a given point of homotopic paths have the same endpoints, it is enough to prove that the path $\gamma_{p(x),v+\dot{x}}$ is homotopic to the concatenation of the path $\gamma_{p(x),v}$ followed by the path $\gamma_{h(p(x),v),w}$.
A homotopy is given by $\Phi(s,t) = h(p(x), \min (\frac{2t}{1+s},1) v + \max(0,\frac{2t+s-1}{1+s}) \dot{x})$. Indeed, $\Phi(s,0)=p(x)$, $\Phi(s,1)=h(p(x),v+\dot{x})$, $\Phi(1,t) = h(p(x),t(v+\dot{x})) = \gamma_{p(x),v+\dot{x}}(t)$ and $\Phi(0,t) = h(p(x), \min (2t,1) v + \max(0,2t-1) \dot{x})$, which is $\gamma_{p(x),v}(2t)$ if $t \leqslant \frac12$ and $h(p(x), v + (2t-1) \dot{x}) = h(h(p(x),v),(2t-1)w) = \gamma_{h(p(x),v),w}(2t-1)$ if $t \geqslant \frac12$.

Finally, if $\psi$ is a deck transformation, then $h^\sharp(\psi(x),T_x\psi(v)) = \tilde{\gamma}_{\psi(x),v}(1)$, and $\psi(h^\sharp(x,v)) = \psi(\tilde{\gamma}_{x,v}(1))$. Both $\tilde{\gamma}_{\psi(x),v}$, and $\tilde{\gamma}_{x,v}$ are lifts of $\gamma_{p(x),v}$, and the second starts at the image by $\psi$ of the origin of the first, so by uniqueness of lifts, it also ends at the image by $\psi$ of the endpoint of the first, which was what to prove.
\end{proof}

The action of the group $\Gamma$ of deck transformations on $M^\sharp$ is free and properly discontinuous, hence tangential. If the covering is normal, then it is also transitive on the fibers and $M^\sharp/\Gamma \simeq M$. The uniqueness result for quotient algebras then implies $(h^\sharp)^\Gamma = h$.

Conversely, if a group $\Gamma$ acts freely and properly discontinuously on a manifold $M$, then $M$ is a (regular) covering space of $M/\Gamma$. If $h$ is an algebra on $M$ and $\Gamma$ acts by automorphisms of $h$, then the uniqueness result for lifted algebras implies $(h^\Gamma)^\sharp = h$.

\section{The tangent functor monad in algebraic geometry}

In this section, we examine how the previous constructions can be carried out in algebraic geometry. For the sake of simplicity, we will restrict our attention to affine schemes, and we will work on the coordinate rings (or algebras) of these affine schemes, which form a category equivalent to the opposite category of affine schemes.

\subsection{The tangent functor}

All rings and algebras in this section are assumed commutative and unital; modules, as well as algebra and ring morphisms and subobjects, are assumed unital. Let $R$ denote a fixed ring.

If $A$ is an $R$-algebra, then the $A$-algebra of K\"{a}hler $R$-differentials of $A$ is denoted by $\Omega_R(A)$. The universal differential is written $d_A \colon A \to \Omega_R(A)$, and the universal property of K\"{a}hler differentials says that
\begin{equation}
\varphi_A \colon \Der_R(A,-) \Longrightarrow \Hom_A(\Omega_R(A),-)
\end{equation}
is a natural isomorphism between these two endofunctors of $\cat{Mod}_A$.

If $M$ is an $R$-module, then the symmetric algebra of $M$ is denoted by $S_R(M)$. Finally, if $A$ is an $R$-algebra, we denote by $U_{A,R} \colon \cat{Mod}_A \to \cat{Mod}_R$ the ``restriction of scalars'' functor.

If we make an analogy with differential geometry, the $R$-algebra $A$ corresponds to the $\RR$-algebra $C^\infty(P)$ of functions on a manifold $P$, and the K\"{a}hler differentials $\Omega_R(A)$ correspond to the De Rham differentials $\Omega^1(P)$. This is not an exact correspondence: the $C^\infty(P)$-linear surjection $\Omega_\RR(C^\infty(P)) \twoheadrightarrow \Omega^1(P)$ is not injective as soon as $\dim P > 0$.
Actually, $\Omega^1(P) = \frakX(P)^*$ as $C^\infty(P)$-modules and $\frakX(P) = \Der_\RR(C^\infty(P))$, but by the universal property of K\"{a}hler differentials, $\Der_R(A) = \Omega_R(A)^*$ as $A$-modules. Therefore, $\Omega^1(P) \simeq \Omega_\RR(C^\infty(P))^{**}$ as $C^\infty(P)$-modules.

Now, $\Omega^1(P)$ is the module of functions on $TP$ which are fiberwise linear. On the algebraic side, one studies polynomial functions, so the symmetric algebra $S_A(\Omega_R(A))$ is a good candidate for the analog of $C^\infty(TP)$. This is in accordance with the general construction of the vector bundle associated to a sheaf of modules: if $X$ is a scheme and $M$ is a locally free sheaf of $\calO_X$-modules, then the relative spectrum of $S_{\calO_X}(M)$ is a (locally trivial) vector bundle on $X$ with space of sections the dual of $M$ (see~\cite[ch.II, ex.5.18]{hart} or~\cite[Prop.5 du n\textordmasculine41]{fac}, or~\cite[ch.II, 1.7]{ega}, where vector bundles are defined by applying this construction to a quasi-coherent sheaf of $\calO_X$-modules, hence are not assumed locally trivial).

Then, in the same way that we ``forgot'' that $TP$ is a vector bundle, we have to forget that $S_A(\Omega_R(A))$ is an $A$-algebra. Therefore, we define the \define{tangent functor} $T \colon \cat{Alg}_R \to \cat{Alg}_R$ by
\begin{equation}
T = U_{\bullet,R} \circ S_\bullet \circ \Omega_R
\end{equation}
on objects, that is, $TA = U_{A,R} ( S_A ( \Omega_R(A)) )$. On morphisms, if $f \colon A \to B$ is an $R$-algebra morphism, then $\Omega_R(f)$ can be regarded as $A$-linear when $\Omega_R(B)$ is seen as an $A$-module via $f$, so $S_A(\Omega_R(f))$ can be regarded as an $A$-algebra morphism, and in particular an $R$-algebra morphism. The algebra $TA$ is called the \define{tangent algebra} of $A$.

As an example, if $A = R[X_1 , \ldots, X_n]$, then $\Omega_R(A) = A \langle dX_1, \ldots, dX_n \rangle$ is the free $A$-module generated by the $dX_i$'s, and $TA = R[X_1 , \ldots, X_n, dX_1, \ldots, dX_n]$, as expected.

\subsection{The tangent functor comonad}

Since we work in the opposite category of affine schemes, we will construct a comonad instead of a monad. The counit is given by
\[
\zeta_A \colon TA \isomto  S_A(\Omega_R(A)) \twoheadrightarrow S_A(\Omega_R(A))_0 \isomto A
\]
given by the projection to the degree-0 term, which is analogous to $\zeta_M \colon M \to TM$.

We also have the analog of the tangent bundle projection,
\[
\tau_A \colon A \isomto S_A(\Omega_R(A))_0 \subseteq S_A(\Omega_R(A)) \isomto TA
\]
given by inclusion. In particular, we have $T\tau_A, \tau_{TA} \colon TA \to T^2A$.

Now we have to look for the analog of fiberwise addition.
On the differential geometry side, if $f, g \colon N \to E$ equalize the vector bundle projection $p \colon E \to M$, then we can define $f+g \colon N \to E$ by fiberwise addition. This is represented by
\[
f + g \colon N \xrightarrow{f \ast g} E \times_p E \xrightarrow{+} E
\]
where $f \ast g$ is the unique function given by the universal property of the pullback.

Dually, the pushout in the category of $A$-algebras is given by the tensor product over the algebra at the ``vertex'' of the pushout square:
\[
\xymatrix{
A \ar[r]^\beta \ar[d]_\gamma & B \ar[d]^{i_B} \\
C \ar[r]_{i_C} & B \otimes_A C
}
\]
We also have to define a coaddition $\boxplus_E \colon S_A(E) \to S_A(E) \otimes_A S_A(E)$ for $E$ an $A$-module. In differential geometry, the pullback by addition on $\Omega^1(M) \simeq C^\infty(TM)_{\mathrm{FibLin}}$ is given by $\omega \mapsto \omega \otimes 1 + 1 \otimes \omega \in \Omega^1(M) \otimes_{C^\infty(M)} \Omega^1(M) \subseteq C^\infty(TM \otimes_M TM)$. Therefore, for any $A$-module $E$, we define
\begin{align*}
\Delta_E \colon E & \to E \otimes_A E\\
x & \mapsto x \otimes 1 + 1 \otimes x.
\end{align*}
We define $j \colon S_A(E \otimes_A E) \hookrightarrow S_A(E) \otimes_A S_A(E)$ to be the inclusion induced by the inclusion $i \otimes_A i \colon E \otimes_A E \hookrightarrow S_A(E) \otimes_A S_A(E)$. Then we define the $A$-algebra morphism
\[
\boxplus_E = j \circ S_A(\Delta_E) \colon S_A(E) \to S_A(E) \otimes_A S_A(E).
\]
In particular, note that $\boxplus_E \neq \Delta_{S_A(E)}$.
If $e_i \in E$ for $i \in I$ a finite set, then
\[
\boxplus_E \left( \prod_{i \in I} e_i \right) =
\sum_{J \subseteq I} \left( \prod_{j \in J} e_j \right) \otimes \left( \prod_{k \in I \setminus J} e_k \right).
\]

If $E$ is an $A$-module and if $f, g \colon S_A(E) \to A'$ are $R$-algebra morphisms which coequalize $i \colon A \to S_A(E)$, we define
\[
f \boxplus_E g \colon S_A(E) \xrightarrow{\boxplus_E} S_A(E) \otimes_A S_A(E) \xrightarrow{f \otimes_A g} A'.
\]
Now we can define $\mu \colon T \Rightarrow T^2$ by
\[
\mu_A = \tau_{TA} \boxplus_{\Omega_R(A)} T\tau_A
\]
for any $R$-algebra $A$.

\begin{prop}
The triple $(T, \zeta, \mu)$ is a comonad in $\mathbf{Alg}_R$.
\end{prop}

\begin{proof}
Let $A$ be an $R$-algebra. It is enough to check the axioms on elements of the form $d_A a \in TA$, since they generate the $A$-algebra $S_A(\Omega_R(A))$ and the maps involved are $A$-algebra morphisms (before restriction of scalars). Also, note that
\[
\mu_A(da) = da + d_T a
\]
where $d$ is short for $d_A$ and $d_T$ is short for $d_{TA}$.

For the counit laws: from $\zeta_{TA}(da) = da$ and $\zeta_{TA}(d_Ta) = 0$, we obtain $\zeta_{TA} \circ \mu_A = \id_{TA}$. Similarly, from $T\zeta_{A}(da) = 0$ and $T\zeta_{A}(d_Ta) = da$, we obtain $T\zeta_{A} \circ \mu_A = \id_{TA}$. 

For coassociativity: $\mu_{TA}(da)=da$ and $\mu_{TA}(d_Ta) = d_Ta + d_{T^2}a$, while $T\mu_{A}(da)=da+d_Ta$ and $T\mu_{A}(d_Ta) = d_{T^2}a$. Therefore, $\mu_{TA} \circ \mu_A (da) = T\mu_{A} \circ \mu_A (da) = da + d_Ta + d_{T^2}a$.
\end{proof}

The coalgebras on an $R$-algebra $A$ for this comonad are $R$-algebra morphisms $h \colon A \to TA$ such that $\zeta_A \circ h = \id_{TA}$ and $Th \circ h = \mu_A \circ h$. In particular, the counit axiom says that $h(a) = a + \text{higher degree terms}$. The study of these coalgebras will be the object of future research.

%% file: chap3.tex
\chapter{The affine case}
\label{chap:aff}

In this chapter, we work in the category of affine manifolds and affine maps. In this category, we can give a characterization of algebras, and in the parallelizable case, of their morphisms.

We also show that there are other monads and comonads in this category, due to the scarcity of morphisms. We classify them, and also the bimonads and Hopf monads that they form. We briefly look at their Hopf modules.

\section{Affine algebras}
\label{sec:affine}

\subsection{Affine algebras on open subsets of affine spaces}

We begin by classifying affine algebras and morphisms on open subsets of affine spaces.

\begin{prop}
\label{prop:affine}
Let $U$ be an open subset of an affine space with direction space $E$ and let $h \colon TU \to U$ be an affine map. Then $h$ is an algebra if and only if there is an endomorphism $A \in \calL(E)$ with $A^2=0$ such that
\begin{equation}
h(x,v) = x + Av
\end{equation}
for all $(x,v) \in TU$.

If $h$ and $k$ are affine algebras respectively on $U$ and $V$ and given by $A \in \calL(E)$ and $B \in \calL(F)$, then an affine map $f \colon U \to V$ is a morphism from $h$ to $k$ if and only if its linear part $\vec{f} \in \calL(E,F)$ intertwines $A$ and $B$, that is,
\begin{equation}
\vec{f} \circ A = B \circ \vec{f}.
\end{equation}
\end{prop}

\begin{rmk}
If the affine map $h$ were defined only on a neighborhood of the zero section and satisfied the two algebra axioms there, the result would still hold, and similarly for an ``algebra morphism'' in that case: this is because all the results needed were derived by differentiating the algebra axioms along the zero section, or setting some variables to be on the zero section.
\end{rmk}

\begin{proof}
Let $h$ be an affine algebra on $U$. That $h$ is affine means that there is a linear map $\vec{h} \in \calL(E \times E, E)$ such that $h(x,v) = h(x,0) + \vec{h}(0,v)$ for all $(x,v) \in TU$.
Define $A = \vec{h} \circ \pr_2 \in \calL(E)$. Since $h(x,0)=x$ by the first axiom for algebras, we have $h(x,v) = x + Av$.
Finally, note that $A = {h_x}'(0)$, so $A^2=0$ by the general result of Chapter~\ref{chap:tfmonad}.
One can also prove that $A^2=0$ directly: the second axiom for algebras reads in this case $x+Av+A(\dot{x}+A\dot{v}) = x+A(v+\dot{x})$, and simplifying, we get $A^2=0$.
The converse is obvious.

Let $f$ be an affine morphism as in the second statement. Its linear part is equal to any of its tangent maps and the result follows from Proposition~\ref{prop:intertw} (the intertwining property of algebra morphisms).
One can also prove this directly: the required diagram commutativity reads $f(x+Av) = f(x)+ B(\vec{f}v)$, so $\vec{f} \circ A = B \circ \vec{f}$.
Again, the converse is obvious.
\end{proof}

\begin{rmk}
Note that if $h \colon TU \to U$ is an affine algebra given by $A \in \calL(E)$, then $U$ is saturated in the direction $\im A$, \textit{i.e.} $U = U+\im A$.
\end{rmk}

\subsection{Algebras in the affine category}

An \define{affine manifold} is a manifold with an atlas whose coordinate changes are affine, and an affine map is a map whose expression in any pair of charts is affine.
Affine manifolds and affine maps form a subcategory of the smooth category which is stable under the tangent functor.
An affine manifold has a canonical flat connection.
If $M$ is an affine manifold and $(x,v) \in TM$, we write $x+v = \exp_x(v)$, when defined, where $\exp$ denotes the exponential map of the flat connection.

Let $M$ be an affine manifold with trivial holonomy modeled on a vector space $E$, let $x_0 \in M$ and $\Phi_0 \colon T_{x_0}M \to E$ be an isomorphism. Then there is a unique vector bundle isomorphism $\Phi \colon TM \to M \times E$ which extends $\Phi_0$ and such that for any chart $(U,\phi)$ of $M$, the linear automorphism $\Phi_x \circ {\phi'(x)}^{-1} \in \calL(E)$ is independent of $x \in U$.
This shows in particular that an affine manifold with trivial holonomy is parallelizable.

From now on, such a choice of a trivialization, or \define{absolute parallelism}, will be implied for all affine manifolds with trivial holonomy.

If $M$ is an affine manifold with trivial holonomy modeled on the vector space $E$, if $A \in \calL(E)$ and $x \in M$, then we still denote by $A$ the endomorphism $ {\Phi_x} \circ A \circ  {\Phi_x}^{-1} \in \calL(T_xM)$.

We will need the following theorem of Auslander and Markus~\cite{ausmark}, which will allow us to reduce the study of affine manifolds to the study of affine manifolds with trivial holonomy.

\begin{thm}[Auslander--Markus~\cite{ausmark}]
Any finite-dimensional real affine manifold has a holonomy covering space, with trivial holonomy, whose group of deck transformations is the holonomy group of that manifold.
\end{thm}

Let $M$ be an affine manifold modeled on the vector space $E$, and $p \colon M^\sharp \to M$ be its holonomy covering space. As above, identify $T_aM^\sharp$ with $T_{p(a)}M$ via $T_ap$. If $A \in \calL(E)$ and $a \in M^\sharp$, then ${\Phi_a} \circ A \circ  {\Phi_a}^{-1} \in \calL(T_{p(a)}M)$. If $(x,v) \in TM$, then if we write ``$Av$'', this implies that all the morphisms ${\Phi_a} \circ A \circ  {\Phi_a}^{-1} \in \calL(T_xM)$ are equal if $p(a)=x$, and we use this morphism to evaluate $Av$.

As for morphisms, an affine morphism $f \colon M \to N$ between manifolds with trivial holonomy has a well-defined linear part, namely $\Psi_{f(x)} \circ T_xf \circ {\Phi_x}^{-1} \in \calL(E,F)$ for any $x \in M$, where $\Phi$ and $\Psi$ are the absolute parallelisms of $M$ and $N$ respectively. We will talk of the linear part of an affine morphism only if it is well-defined.

We can now state the main theorem of this chapter.

\begin{thm}[Affine algebras and morphisms]
\label{thm:affine}

Let $M$ be an affine manifold modeled on the vector space $E$, either with trivial holonomy or real and finite-dimensional, and let $h \colon TM \to M$ be an affine map. Then $h$ is an algebra if and only if there is an endomorphism $A \in \calL(E)$ with $A^2=0$ such that
\begin{equation}
h(x,v) = x + Av
\end{equation}
for all $(x,v) \in TM$.

If $h$ and $k$ are affine algebras respectively on $M$ and $N$ and given by $A \in \calL(E)$ and $B \in \calL(F)$, and if $M$ and $N$ have trivial holonomy, then an affine map $f \colon M \to N$ is a morphism from $h$ to $k$ if and only if its linear part $\vec{f} \in \calL(E,F)$ intertwines $A$ and $B$.
\end{thm}

\begin{proof}
First we prove the result for manifolds with trivial holonomy.
Let $h$ be an affine algebra on $M$ with $M$ having trivial holonomy. Then on each chart $(U,\phi)$ of $M$, $h$ has the expression given by Proposition~\ref{prop:affine} (because of the remark following that proposition), with an endomorphism $A_U \in \calL(E)$. Since these expressions agree on intersections of charts, the $A_U$'s are actually independent of $U$.

As for morphisms, an affine morphism $f$ between manifolds with trivial holonomy has a well-defined linear part, and it satisfies the intertwining property for the same reason as in Proposition~\ref{prop:affine}.
The converse is obvious.

Now, if $M$ is an affine manifold and $h$ is an algebra on $M$, then there is a lifted algebra on the holonomy covering space $h^\sharp$ which is also affine, hence of the given form. As noted in the section on lifted algebras, we have $h = (h^\sharp)^\Gamma$ where $\Gamma$ is the group of deck transformations of $M^\sharp$. That is, if $(x,v) \in TM$ , then $h(x,v) = p(h^\sharp(a,v))$ for any $a \in M^\sharp$ projecting to $x$. Since $p$ is affine, this reads $h(x,v) = x + A_a v$ for any $a$ projecting to $x$, so by the conventions we made just before the theorem, we can write $h(x,v) = x + Av$.
The converse is obvious.
\end{proof}

An affine algebra of maximal rank on an open subset of an affine space is isomorphic to a free algebra. Indeed, let $h \colon TU \to U$ be of maximal rank, with $D = \im A = \ker A$ and let $F$ be a topological complement of $D$, so that $E = D \oplus F$.
Let $x_0 \in U$, and define $N = (x_0 + F) \cap U$, so that $U \simeq N \times D$. We write $\tilde{A} \colon TN \isomto N \times F \isomto N \times E/D \isomto N \times D \isomto U$. Then $h$ is isomorphic to $\mu_N$ via $\tilde{A}$, meaning that the following diagram commutes:
\[
\xymatrix{
T^2N \ar[r]^{T\tilde{A}} \ar[d]_{\mu_N} & TU \ar[d]^h\\
TN \ar[r]_{\tilde{A}} & U
}
\]
Indeed, if $(x,v,\dot{x},\dot{v}) \in T^2N$, then
\[
(x,v,\dot{x},\dot{v}) \xrightarrow{T\tilde{A}} (x+Av,\dot{x}+A\dot{v}) \xrightarrow{h} x+Av+A\dot{x}
\quad\text{and}\quad
(x,v,\dot{x},\dot{v}) \xrightarrow{\mu_N} (x,v+\dot{x}) \xrightarrow{\tilde{A}} x+A(v+\dot{x}).
\]

This does not hold for general affine algebras on parallelizable manifolds. Indeed, the leaves of a free algebra are the tangent spaces, and therefore are noncompact (in positive dimension). The affine algebra $h \colon TM \to M$ where $M =\RR \times \mathbb{S}^1$ is the cylinder, given by $A = \begin{pmatrix} 0&0\\1&0 \end{pmatrix}$ has maximal rank, but its leaves are the circles $\{a\} \times \mathbb{S}^1$, which are compact. Note however that $h$ is the quotient algebra $(\mu_\RR)^\ZZ$ of the free algebra $\mu_\RR$ by the group $\ZZ$ acting by addition on the fibers of $T\RR$.

\subsection{Semi-affine algebras}

To give a few more examples of algebras, we consider the slightly more general ``semi-affine'' case of maps $h \colon TU \to U$, where $U$ is an open subset of an affine space of direction space $E$. These are the algebras of the form
\begin{equation}
h(x,v) = x + A_x v
\end{equation}
where $A \colon U \to \calL(E)$ is a smooth function, which we consider as a $(1,1)$-tensor field in the finite-dimensional case.

\begin{prop}
Let $U$ be an open subset of a finite-dimensional affine space.
Let $A$ be a $(1,1)$-tensor field on $U$ and let $h \colon TU \to U$ be defined by
\begin{equation*}
h(x,v) = x + A_x v.
\end{equation*}
for $(x,v) \in TU$. If $h$ is an algebra, then $A^2=0$ and the Nijenhuis tensor of $A$ vanishes.
\end{prop}

Recall that the Nijenhuis tensor of a $(1,1)$-tensor $A$ is the $(2,1)$-tensor defined by
\begin{equation}
N_A(X,Y) = [AX,AY] - A \big( [AX,Y] + [X,AY] - A[X,Y] \big)
\end{equation}
for any vector fields $X, Y \in \frakX(M)$.
Writing $[X,Y] = X \cdot Y - Y \cdot X$, we have
\begin{align*}
N_A(X,Y) & = AX \cdot AY - AY \cdot AX\\
&\qquad - A \big( AX \cdot Y - Y \cdot AX + X \cdot AY - AY \cdot X - A(X \cdot Y - Y \cdot X) \big)\\
& = \big( AX \cdot AY - A(AX \cdot Y) \big) - \big( AY \cdot AY - A(AY \cdot X) \big)\\
&\qquad - A \big( X \cdot AY - A(X \cdot Y) \big) + A \big( Y \cdot AX - A(Y \cdot X) \big)
\end{align*}
and using the relation $Z \cdot BT = (Z \cdot B)T + B(Z \cdot T)$ for any vector fields $Z$ and $T$ and  any $(1,1)$-tensor $B$, we obtain
\begin{equation*}
N_A(X,Y) = (AX \cdot A)Y  - (AY \cdot A)X - A((X \cdot A)Y) + A((Y \cdot A)X).
\end{equation*}

\begin{proof}
That $A^2=0$ was already proved.
The second axiom for algebras with $\dot{v}=0$ gives $A_{x+A_x v} (\dot{x} + A'_x(\dot{x})v ) = A_x \dot{x}$. Differentiating with respect to $v$ at $v=0$ gives
\[
A'_x \big( A_x(v') \big) (\dot{x}) + A_x \big( A'_x(\dot{x})v' \big) = 0.
\]
Therefore, if $X$ and $Y$ are vector fields extending $\dot{x}$ and $v'$ respectively, we have
\[
(AY \cdot A)X + A ( (X \cdot A)Y ) = 0
\]
and similarly with $X$ and $Y$ exchanged,
which implies that $N_A=0$.
\end{proof}

If $h$ is as in the proposition with $A_x=f(x)B$ where $f \in C^\infty(U)$ and $B^2=0$, then $h$ is an algebra if and only if $f$ is constant on each translate of $\im B$, as the second algebra axiom shows. Combining this with the quotient construction gives more examples of algebras, for instance, we can define an algebra on $\RR^n \times \TT^n$ by
\[
h(x,\theta,\dot{x},\dot{\theta}) = (x,\theta+f(x)p(\dot{x}))
\]
for any function $f \in C^\infty(\RR)$, where $p \colon \RR^n \to \TT^n$ is the quotient map. The leaves of $h$ are tori and they are ``travelled'' at different speeds, parametrized by the other factor, $\RR^n$.\\

An \define{almost tangent structure} on a manifold is a $(1,1)$-tensor field $A$ such that $\im A = \ker A$.
It is \define{integrable} if its Nijenhuis tensor vanishes.
Therefore, we have proved the following (in finite dimension):

\begin{prop}
The associated endomorphism of a semi-affine algebra of maximal rank is an integrable almost tangent structure.
\end{prop}

\section{Hopf monads and their Hopf algebras}

\subsection{Comonads in the affine category}

We have seen above that there is only one monad and no comonad on the tangent functor in the categories of smooth and analytic manifolds. However in the category of affine manifolds, there are much ``fewer'' morphisms, and therefore being natural is less demanding.

The zero section $\zeta \colon \id_{\cat{AffMan}} \Rightarrow T$ is still the only natural transformation between these functors, by Lemma~\ref{lem:uniqNat}. Similar arguments show that the only natural transformations from $T$ to $\id_{\cat{CompAffMan}}$ are the ``projections'' $\tau^t$ with $t \in \KK$ defined in charts by
\[
\tau^t_M \colon (x,v) \mapsto x + tv,
\]
the natural transformations from $T^2$ to $T$ are of the form $(x,v,\dot{x},\dot{v}) \mapsto (x+a'v+b'\dot{x}+c'\dot{v},av+b\dot{x}+c\dot{v})$ and the natural transformations from $T$ to $T^2$ are of the form $(x,v) \mapsto (x+a'v,av,bv,cv)$. In the category of (not necessarily complete) affine manifolds, we must impose $t=a'=b'=c'=0$. Using the (co)unit laws and the (co)associativity axioms, we obtain the following result.

\begin{prop}
In the categories of affine manifolds and of complete affine manifolds, the only monads on the tangent functor are the $(T,\zeta,\mu^a)$ where
\[
\mu^a_M \colon (x,v,\dot{x},\dot{v}) \mapsto (x,v+\dot{x}+a\dot{v})
\]
for any $a \in \KK$, and the only comonads on the tangent functor are the $(T,\tau,\delta^b)$ where
\[
\delta^b_M \colon (x,v) \mapsto (x,v,v,bv)
\]
for any $b \in \KK$.
\end{prop}

The comultiplications $\delta^b$ associate naturally to each affine manifold a second-order vector field, or semispray. This semispray is a spray if and only if $b = 0$, in which case it is the geodesic spray.

A coalgebra on the manifold $M$ over this comonad is a map $X \colon M \to TM$ such that $\tau_M \circ X = \id_M$ and $TX \circ X = TX \circ \delta^b_M$. The first condition says that $X$ is a vector field on $M$, so it is of the form $X(x) = (x, X_x)$ and the second condition reads in charts $X'_x(X_x) = b X_x$.

In a chart $(U,\phi)$, $X$ is of the form $X_x = X_{x_0} + B (x-x_0)$ with $x_0 \in U$ and $b\in E$ and $B \in \calL(E)$. Then, coassociativity implies
\[
BX_{x_0} = b X_{x_0}
\quad\text{and}\quad
B^2 = b B
\]
and similarly, algebras are now of the form $h(x,v) = x + Av$ with $A^2 = aA$.

\subsection{Hopf monads in the affine category}

In this subsection, we combine the tangent functor monads and comonads, to form Hopf monads.
The notions of bimonad and Hopf monad, defined in~\cite{hopf}, are analog to those of bialgebra and Hopf algebra.

\begin{thm}[Tangent functor Hopf monads]
The quintuple
\begin{equation}
(T,\zeta, \mu^a,\tau,\delta^b)
\end{equation}
for $a, b \in \KK$ is a bimonad in $\cat{AffMan}$. In both $\cat{AffMan}$ and $\cat{CompAffMan}$, it is a Hopf monad if and only if $1+ ab \neq 0$, and the antipode is then the natural transformation $\sigma = m_{-(1+ab)^{-1}}$ of $T$, the fiberwise multiplication by $-\frac1{1+ab}$.
In particular, if $ab=0$, the antipode is the fiberwise opposite.
\end{thm}

\begin{proof}
From the definition of a bimonad in~\cite[Def.4.1]{hopf}, we have to prove the commutativity of the three diagrams
\[
\xymatrix{
T^2 \ar[r]^{T\tau} \ar[d]_{\mu^a} & T \ar[d]^{\tau}\\
T \ar[r]_{\tau} & \id
}
\qquad
\xymatrix{
\id \ar[r]^{\zeta} \ar[d]_{\zeta} & T \ar[d]^{\delta^b}\\
T \ar[r]_{\zeta T} & T^2
}
\qquad
\xymatrix{
\id \ar[r]^{\zeta} \ar[dr]_{\id} & T \ar[d]^{\tau}\\
& T
}
\]
which mean that $\tau$ is a monad morphism from $(T,\zeta,\mu^a)$ to the identity and $\zeta$ is a comonad morphism from the identity to $(T,\tau,\delta^b)$. The  commutativity is obvious. We also have to prove the existence of a natural transformation (a ``mixed distributive law'', or ``entwining'') $\lambda \colon T^2 \Rightarrow T^2$ such that the diagram
\[
\xymatrix{
T^2 \ar[r]^{\mu^a} \ar[d]_{T\delta^b} & T \ar[r]^{\delta^b} & T^2\\
T^3 \ar[rr]_{\lambda T} && T^3 \ar[u]_{T\mu^a}
}
\]
commutes.
Setting $\lambda_M \colon (x,v,\dot{x},\dot{v}) \mapsto (x,\dot{x},v+\dot{x}+a\dot{v}, b\dot{x} - \dot{v})$ fulfills this condition:
\[
(x,v,\dot{x},\dot{v})
\xrightarrow{\mu^a_M}
(x,v+\dot{x}+a\dot{v})
\xrightarrow{\delta^b_M}
(x,v+\dot{x}+a\dot{v},v+\dot{x}+a\dot{v},b(v+\dot{x}+a\dot{v}))
\]
while
\begin{multline*}
(x,v,\dot{x},\dot{v})
\xrightarrow{T\delta^b_M}
(x,v,v,bv,\dot{x},\dot{v},\dot{v},b\dot{v})
\xrightarrow{\lambda_{TM}}
(x,v,\dot{x}, \dot{v}, v+\dot{x}+a\dot{v}, bv+\dot{v}+ab\dot{v}, b\dot{x}-\dot{v}, 0 ) \\
\xrightarrow{T\mu^a_M}
(x,v+\dot{x}+a\dot{v},v+\dot{x}+a\dot{v}, bv + b\dot{x} + ab\dot{v}).
\end{multline*}

If it exists, an antipode is a natural transformation $\sigma \colon T \Rightarrow T$ such that
\[
\mu^a \circ \sigma T \circ \delta^b = \mu^a \circ T\sigma \circ \delta^b = \zeta \circ \tau.
\]
A natural transformation $\sigma \colon T \Rightarrow T$ has the form $\sigma_M(x,v) = (x+sv,tv)$ for $s,t \in \KK$.
Therefore $(x,v,v,bv) \xrightarrow{\sigma_{TM}} (x+sv,(1+bs)v,tv,tbv) \xrightarrow{\mu^a_M} (x+sv,(1+bs+t+abt)v)$,
so we need $s=1+bs+t+abt=0$, so $s=0$ and $1+(1+ab)t=0$. So we need $1+ab \neq 0$ and we set $t=-\frac1{1+ab}$. Then, we also have
$(x,v,v,bv) \xrightarrow{T\sigma_M} (x,tv,v,btv) \xrightarrow{\mu^a_M} (x,(1+t+abt)v)=(x,0)$.
\end{proof}

A \define{Hopf module} for this bimonad is a pair $(h,X)$ where $h \colon TM \to M$ is an algebra and $X \colon M \to TM$ is a coalgebra, which are compatible in the following sense (see~\cite[Def.4.2]{hopf}): the diagram
\[
\xymatrix{
TM \ar[r]^h \ar[d]_{TX} & M \ar[r]^X & TM\\
T^2M \ar[rr]_{\lambda_M} && T^2M \ar[u]_{Th}
}
\]
should commute. If $h$ and $X$ are given by $A, B \in \calL(E)$ respectively, then this is equivalent to
\begin{equation}
\label{eq:hopfMod}
(A-a)B + (B-b)A = \id_E
\end{equation}
in addition to $A^2 = aA$ and $B^2 = bB$.\\

As an example, we determine the Hopf modules $(h,X)$ on $\KK^2$, where $h$ is given by $A$ and $X$ is given by $X_0=X_{(0,0)}$ and $B$.
First, we suppose that $a=0$.
Then $A \neq 0$ by~\eqref{eq:hopfMod}, so we can suppose, up to a change of basis, that $A=\begin{pmatrix}0&1\\0&0\end{pmatrix}$.
Then $B=\begin{pmatrix}t&t(b-t)\\1&b-t\end{pmatrix}$ for some $t \in \KK$ and $X_0 \in \KK \vect{t}{1}$.
If $a \neq 0$, then there are three cases: $A=0$ or $A=a\id$ or up to a change of basis, $A=\begin{pmatrix}a&0\\0&0\end{pmatrix}$.
If $A=0$, then $ab+1=0$ and $B=b \id$, so $X_{(0,0)}$ can be any vector.
If $A=a \id$, then $ab+1=0$ and $X=0$.
If $A=\begin{pmatrix}a&0\\0&0\end{pmatrix}$, then: if $ab+1=0$, then
$B=\begin{pmatrix}0&0\\t&b\end{pmatrix}$ for some $t \in \KK$ and $X_0 \in \KK \vect{0}{1}$, else
$B=\begin{pmatrix}\frac1a+b&\frac1a+b\\-\frac1a&-\frac1a\end{pmatrix}$ and $X_0 \in \KK \vect{ab+1}{-1}$.
Conversely, these pairs $(h,X)$ are Hopf modules.

%% file: chap4.tex
\chapter{The foliation induced by an algebra}
\label{chap:foli}

In this chapter, we prove our main structure theorem, which links the study of algebras over the tangent functor monad to the study of foliations. To do this, we introduce the accessible sets of an algebra and show that they form a foliation. The dimensions of the leaves of these foliations may vary, so these are singular, or \v{S}tefan foliations. Extensions of the classical theorem of Fr\"{o}benius were found by \v{S}tefan and Sussmann among others, and we review the necessary elements of their work in Section~\ref{sec:stefan}. We then prove the main theorem, stating that tame algebras induce foliations. This gives rise to the notion of monadic foliation: a foliation which arises in this way from an algebra. We give a necessary condition for a foliation to be monadic in terms of its holonomy, where we use a result akin to control theory. Finally, we see how to reduce the study of algebras to the study of orientable ones.

\section{Accessible sets of an algebra}

If $h$ is an algebra on the manifold $M$ and $x \in M$, we define the \define{accessible set}, or \define{leaf} of $x$ for the algebra $h$ to be the subset
\begin{equation}
L_x = h_x(T_xM)
\end{equation}
of $M$. The accessible sets are path-connected as continuous images of path-connected spaces. Algebra morphisms map accessible sets into accessible sets. We begin by proving the following easy set-theoretic result.

\begin{prop}
The accessible sets of an algebra form a partition of the base manifold.
\end{prop}

\begin{proof}
The accessible sets obviously cover $M$ since for all $x \in M$, $x \in L_x$. Therefore we are left to show that $L_x \cap L_y \neq \varnothing$ implies $L_x = L_y$ for any $x,y \in M$, and we will do it in three steps.
First, $z \in L_x$ implies $L_x \subseteq L_z$.
Indeed, $z \in L_x$ means that there is a $v \in T_xM$ such that $z=h(x,v)$, and if $x_1 \in L_x$, then there is a $v_1 \in T_xM$ such that $h(x,v_1)=x_1$. Choose a $\xi \in T_{(x,v)}TM$ such that $T\tau_M(\xi)=v_1-v$. Then
\begin{equation*}
x_1 = h(x,v_1) = h \big( x,v+(v_1-v) \big) = h \big( h(x,v),h'(x,v)(\xi) \big) = h \big( z,h'(x,v)(\xi) \big) \in L_z.
\end{equation*}
Second, $z \in L_x$ implies $L_x = L_z$. Indeed, by the above, if $z \in L_x$, then $L_x \subseteq L_z$, so in particular $x \in L_z$, so $L_z \subseteq L_x$, so $L_x = L_z$.
Third, if $L_x \cap L_y \neq \varnothing$, then there is a $z \in L_x \cap L_y$, so $L_x = L_z = L_y$.
\end{proof}

We are going to prove that much more is true: the accessible sets are initial submanifolds of $M$ that ``fit together nicely'', in other words, they form a foliation. Since the leaves may have different dimensions, this type of foliation is sometimes called a singular foliation, or \v{S}tefan foliation.

\section{Distributions, foliations, integrability}
\label{sec:stefan}

First, we recall the notion of a (singular) distribution. These distributions need not have constant rank, so, even when smooth, they need not be subbundles of the tangent bundle.

\begin{defi}
A \define{distribution} $\calD$ on a manifold $M$ is the data of a complemented vector subspace of $T_xM$ at each point of $x \in M$.
A local vector field $X \in \frakX_U(M)$ is said to belong to $\calD$ if for any $x \in U$, $X_x \in \calD_x$. In that case, we write $X \in \frakX_U(\calD)$.
\end{defi}

In other words, a distribution on a manifold $M$ is a (set-theoretic) section of the Grassmannian bundle $\mathrm{Gr}(TM) \to M$.
The ``distribution'' induced by an algebra is a genuine distribution as soon as the algebra is tame or has pointwise finite rank.

The \define{rank} of the distribution $\calD$ at $x$ is the dimension, finite or infinite, of $\calD_x$.
The rank at a point of the distribution induced by an algebra is the rank of the algebra at that point.

\begin{defi}[Smooth distribution, see~\cite{nubel}]
A distribution is \define{smooth} if it is locally supported by vector subbundles. Explicitly, a distribution $\calD$ on $M$ is smooth if for any $x \in M$ there is an open neighborhood $U$ of $x$ and a vector subbundle $B$ of $TU$ such that $B \subseteq \calD_U$ and $B_x = \calD_x$.
\end{defi}

Being smooth implies, and in the case of pointwise finite rank is equivalent to, being generated by its local vector fields, meaning: for any $v \in \calD_x$, there exists a local vector field $X \in \frakX_U(\calD)$ with $X_x=v$. In a sense, ``smooth'' can be thought of as ``lower semiregular''.
As an example, consider the distributions on $\RR$ given by $\calD_x = \spann(x e_1)$ and $\calD'_x = \ker(x \,dx)$.
Then $\calD$ is smooth but $\calD'$ is not.

A \define{regular} distribution is a vector subbundle of $TM$.
In particular, it is smooth.
For distributions of pointwise finite rank, being regular is equivalent to being smooth and having locally constant rank.

We now turn to the integrability of foliations. In the regular case, this is equivalent to involutivity under the Lie bracket: this is the classical Fr\"{o}benius theorem.
Involutivity is not sufficient anymore in the singular case, and integrability theorems in the singular case were given by Sussmann (\cite{suss}) and \v{S}tefan (\cite{ste74b,ste74,ste80}).
There are several equivalent definitions of a singular foliation. Here is one (see also Michor~\cite[Ch.I 3.18--3.28]{mich}, or~\cite[\S8.2]{prad} for an emphasis on lower semiregularity).

\begin{defi}
A \define{foliation} of a manifold $M$ is a partition of $M$ into connected initial submanifolds $S_x$ called ``leaves'', such that for every $x \in M$, there is a chart $(U,\phi)$ such that $\phi(U)= V \times W$ with $V$ (resp. $W$) a neighborhood of 0 in a Banach space $E$ (resp. $F$), $\phi(x)=(0,0)$, and if $S$ is a leaf, then $\phi(U \cap S) = V \times A$ for some $A \subseteq W$.
\end{defi}

\begin{rmk}
In the definition of a foliation, it is sufficient, in finite dimension, to suppose that the leaves form a partition into connected immersed submanifolds: that they are initial submanifolds is a consequence of the other conditions, although less easy than in the regular case (see~\cite{kuba}).
\end{rmk}

Actually, \v{S}tefan's work not only gives necessary and sufficient conditions for integrability of smooth distributions, but it constructs a foliation from \emph{any} smooth distribution. This foliation then has a tangent distribution, larger than the initial distribution, and some conditions for equality can be given.

If $\calD$ is a smooth distribution on a manifold $M$, we denote by $\Psi$ the set of local diffeomorphisms of $M$ which are the composition of a finite number of time-1 flows of local vector fields in $\calD$.
We denote by $\Psi_{x \to y}$ the set of those local diffeomorphisms which send $x$ to $y$.
This defines an equivalence relation, namely $x$ equivalent to $y$ if $\Psi_{x \to y} \neq \varnothing$.
The equivalence class of $x$ is denoted by $S_x$ and is called the \define{orbit}, or \v{S}tefan-accessible set, of $x$.

We can now state the following theorem, due to \v{S}tefan~\cite[Theorem~1]{ste74}, which will be a key ingredient in our proof of the main theorem.

\begin{thm}
The orbits of any smooth distribution $\calD$ are the leaves of a foliation with associated distribution $\check{\calD}$, where
\begin{equation*}
\check{\calD}_x = \sum_{y \in S_x, \psi \in \Psi_{y \to x}} T_y\psi (\calD_y)
\end{equation*}
for any $x \in M$.
\end{thm}

We also define another partition of $M$ following~\cite{nubel}.
A \define{$\calD$-curve} is a curve $\gamma \colon I \to M$ from an open interval of $\RR$ such that $\dot{\gamma}(t) \in \calD_{\gamma(t)}$ for all $t \in I$ and that the pullback $\gamma^* \calD$ is a subbundle of $\gamma^* TM$.
We define the equivalence relation: being connected by a finite number of $\calD$-curves.
The equivalence class of $x$ is denoted by $N_x$ and is called the \define{reachable set}, or N\"{u}bel-accessible set, of $x$.

\begin{rmk}
If we remove for $\calD$-curves the condition involving pullbacks (which in the case of pointwise finite rank is equivalent to constant rank along $\gamma$), then we obtain a partition obviously coarser than that by the orbits of $\calD$, which need not form a foliation even in the simplest cases, as the distribution defined by the vector field in the plane $X \colon (x,y) \mapsto (x, -y)$ shows.
Also, the distribution on $\RR$ given by $\calD_x = \spann(x e_1)$ yields the trivial equivalence relation.
\end{rmk}

Here is an example from~\cite{nubel} that shows that the N\"{u}bel-accessible sets can differ from the orbits, and contrary to those, need not be immersed submanifolds (even less form a foliation): consider the foliation of the plane $(x,y)$ given by the whole tangent spaces for $x>0$ and the ``horizontal'' subspaces for $x \leqslant 0$. Then the N\"{u}bel-accessible sets are the open half-plane $\{x>0\}$ and the horizontal closed half-lines for $x \leqslant 0$.

\section{The foliation induced by an algebra}

The \v{S}tefan-accessible sets of the distribution induced by an algebra are \textit{a priori} different from the accessible sets defined at the beginning of this chapter. We will see that they are actually equal, and therefore are the leaves of a foliation of the base manifold.

First we mention the following characterization of tame algebras.

\begin{prop}
If an algebra is tame, then the rank of its associated distribution is constant on its accessible sets.
The converse holds for pointwise finite rank.
\end{prop}

\begin{proof}
If $x$ and $y$ are in the same accessible set, we can write $y=h(x,v)$. Then the function $[0,1] \to \NN, t \mapsto \dim \calD_{h(x,tv)}$ is locally constant by~\eqref{eq:incl-im} (case of equality), hence constant.
The converse also follows from~\eqref{eq:incl-im} (case of equality) and the fact that a map of locally constant finite rank is a subimmersion.
\end{proof}

To apply the results of the previous section, we need the distributions we consider to be smooth. This is the object of the following proposition.

\begin{prop}
\label{prop:smooth}
The distribution induced by an algebra is smooth as soon as the algebra is tame or the ``distribution'' is a genuine distribution on a Hilbert manifold, or it has pointwise finite rank.
\end{prop}

\begin{proof}
Let $h$ be an algebra on $M$ with induced distribution $\calD$. In the case of pointwise finite rank, we can proceed as follows. Let $x \in M$ and let $v \in \calD_x$. Then there is a $w \in T_xM$ with $A_x(w) = v$. Extend $w$ to a vector field $W$. Then the vector field $V \colon y \mapsto A_y(W_y)$ extends $v$ and belongs to $\calD$.

For the general case, if $x \in M$, then our hypotheses (either tameness or genuine distribution on a Hilbert manifold), ensure that both the kernel and the image of $A_x$ split, so we can write $A_x \colon \ker A_x \oplus B_x \to \calD_x \oplus C_x$. Then, from~\cite[Observation p.~3]{nubel}, there is a neighborhood of $x$ and a subbundle $B$ extending $B_x$, such that $A$ is an isomorphism of vector bundles from $B$ onto the subbundle $A(B)$, which locally supports $\calD$.
\end{proof}

We will need the following lemma (very similar to the above proposition) in the proof of the main theorem.

\begin{lem}
The algebra-accessible sets of a tame algebra form a partition finer than the N\"{u}bel partition.
\end{lem}

\begin{proof}
Given a tame algebra $h$, we have to prove that the $h$-accessible sets are included in the N\"{u}bel-accessible sets.
Let $x \in M$ and $y \in L_x$.
Then there is a $v \in T_xM$ such that $y = h(x,v)$ and we define the curve $\gamma \colon \RR \to M, t \mapsto h_x(tv)$.
Then $\gamma'(t) = {h_x}'(tv)v \in \im {h_x}'(tv) \subseteq \calD_{h_x(tv)}$ and $\gamma(0) = x$ and $\gamma(1) = y$.

If $t \in \RR$, then $h_{\gamma(t)}$ is a subimmersion in a neighborhood of the origin, so by~\eqref{eq:incl-im} (case of equality), $\gamma^* \calD$ is a subbundle of $\gamma^* TM$.
Therefore $\gamma$ is a $\calD$-curve connecting $x$ and $y$.
\end{proof}

From now on, manifolds are assumed to be modeled on separable Banach spaces.
We can now prove the main result of this chapter.

\begin{thm}\label{thm:folia}
The accessible sets of a tame algebra are the leaves of a foliation with tangent distribution the distribution induced by that algebra, which is therefore integrable.
\end{thm}

\begin{proof}
Let $h$ be a tame algebra on the manifold $M$ and let $\calD$ be its induced distribution, which is smooth by Proposition~\ref{prop:smooth}.
We are going to prove that the \v{S}tefan partition is finer than the partition into $h$-accessible sets.
Then we will prove that $\calD$ is integrable.
This will imply that the N\"{u}bel and \v{S}tefan partitions agree (by~\cite[Theorem~4.ii]{nubel}), and by the lemma that they also agree with the partition into $h$-accessible sets, which is therefore a foliation.

Let $x \in M$.
The set $S_x$ is the set of points obtained from $x$ by following a finite number of integral curves of local vector fields in $\calD$.
We want to show that $S_x \subseteq L_x$.
Since the $L_y$'s form a partition, it is enough to show that the set of points obtained from $x$ by following one integral curve of a local vector field in $\calD$ is included in $L_x$ (we use transitivity: if $y \in L_x$ and $z \in L_y$ then $z \in L_x$). Similarly, given such a curve $\gamma$ with $\gamma(0)=x$, it is enough to show that there is a $\delta>0$ with $\gamma(-\delta,\delta) \subseteq L_x$: indeed, if $\gamma$ has a limit at $\delta$, say $y$, then $\gamma$ is defined on an open neighborhood of $\delta$ and we apply the same argument at $\gamma$ around $y$ and obtain that $L_x \cap L_y \neq \varnothing$ hence $L_x = L_y$.

Let $X \in \Gamma_U(\calD), U \subseteq M$ be a non-vanishing local vector field in $\calD$. Fix $x \in U$ and write $f = h_x$.
Since $h$ is tame at $x$, that is, $f$ is a subimmersion at 0, there is a closed subspace $F \subseteq T_xM$ and an open neighborhood $U_x$ of 0 in $T_xM$ such that for $u \in U_x$, one has $T_xM = \ker f'(u) \oplus F$, so that $f'(u)$ induces an isomorphism $F \isomto \calD_{f(u)}$ whose inverse we denote by $g(u)$. Then the function $k \colon U_x \to F$ defined by
\[
k(u) = g(u)(X_{f(u)})
\]
is smooth. Therefore, the Cauchy problem $v' = k(v), v(0)=0$ has a local solution. For this solution,
\[
\frac{d}{dt} f(v(t)) = f'(v(t))(v'(t)) = f'(v(t)) \big( k(v(t)) \big) = X_{f(v(t))}
\]
so $f \circ v$ is (near $x$) the integral curve of $X$ through $x$, which is therefore included (near $x$) in $L_x$. 

Now we prove that $\check{\calD} = \calD$. By tameness, for any $x \in M$, there is an open neighborhood $U_x \subseteq T_xM$ of 0 where $h_x$ is a subimmersion, so $h_x(U_x)$ is an immersed submanifold of $M$.
Fix $x_0 \in M$ and let $v \in T_{x_0}M$.
Then by the second algebra axiom, there is an open neighborhood $V_v \subseteq T_{x_0}M$ of $v$ such that $h_{x_0}(V_v) \subseteq h_{h(x_0,v)} (U_{h(x_0,v)})$.
The open subsets $V_v$ cover $T_{x_0}M$, which is Lindel\"{o}f (this is equivalent to separable for metrizable spaces), so we can extract a countable cover, and write $L_{x_0} = \bigcup_{i \in \NN} h_{h(x_0,v_i)}(U_{h(x_0,v_i)})$.
Therefore $L_{x_0}$ is a countable union of immersed submanifolds, all modeled on $\calD_{x_0}$ (actually on some $\calD_{h(x_0,v_i)}$ but we saw in the proof of the lemma that they are isomorphic).
In particular, the tangent space of $S_{x_0} \subseteq L_{x_0}$ at $x_0$, which by definition is $\check{\calD}_{x_0}$, is included in a countable union of spaces isomorphic to $\calD_{x_0}$ so cannot be a proper superset of $\calD_{x_0}$, so $\check{\calD} = \calD$.

Since $\calD$ is integrable, \cite[Theorem~4.ii]{nubel} implies that the N\"{u}bel-accessible sets are its (maximal connected) integral manifolds, so agree with the \v{S}tefan-accessible sets.
Therefore by the lemma, both agree with the partition into $h$-accessible sets, which is therefore a foliation with tangent distribution $\calD$.
\end{proof}

If we do not require tameness, it is not true in general that one can find such a continuous function $k$. For instance, take $\lambda_u = \begin{pmatrix}1&0\\u&u^2\end{pmatrix}$ for $u \in \KK$. Then if $u \neq 0$, ${\lambda_u}^{-1}\begin{pmatrix}1\\0\end{pmatrix} = \begin{pmatrix}1\\-\frac1u\end{pmatrix}$ which is unbounded when $u \to 0$; however, $\begin{pmatrix}1\\0\end{pmatrix} \in \im \lambda_0$.

\section{Holonomy of a monadic foliation}
\label{sec:holo}

A foliation which arises as the foliation induced by a tame algebra is called \define{monadic}. Being monadic is obviously a differential invariant, but we will see in the next chapter that it is not a topological invariant. By the nilpotency property of the associated endomorphism of an algebra, a monadic foliation has rank at most half the dimension of the manifold. In this section, we are going to prove another restriction on monadic foliations, regarding their holonomy. For the concept of holonomy of a regular foliation, see for instance~\cite{moer}.

First we need a result reminiscent of control theory, about lifting paths (from the base manifold to the ``parameter'' space).

\begin{prop}[lifting of paths]
\label{pathLift}

Let $h \colon TM \to M$ be a tame algebra, $x \in M$, and $\gamma \colon [0,1] \to M$ a smooth path in $L_x$ with $\gamma(0)=x$. Then there is a piecewise smooth path $\tilde{\gamma} \colon [0,1] \to T_xM$ such that $\tilde{\gamma}(0)=0$ and $h(x,\tilde{\gamma}(t)) = \gamma(t)$ for $t \in [0,1]$.
\end{prop}

Of course, such a lifting is not unique.

\begin{proof}

First, we define a function $f \colon [0,1] \to \RR$ as follows.
If $t \in [0,1]$, let $f(t)$ be half the supremum of the $\nu \in [0,1]$ such that $h_{\gamma(t)}$ has a local section sending $\gamma(t)$ to 0 and defined on an open set containing $\gamma([t,t+\nu])$, and such that for any $v$ in the image of this local section, $h'(\gamma(t),v)$ is surjective.
The function $f$ is well-defined and positive since each $h_{\gamma(t)}$ corestricted to the initial submanifold $L_x$ is submersive at $0$ (because $h$ is tame), hence locally a submersion onto $L_x$, and $h$ is a submersion along the zero section.
It is lower semicontinuous, so has a positive infimum $\eta$ on the compact set $[0,1]$. Set $n = \lceil \frac1\eta \rceil$ and $x_1 = \gamma(\frac1n)$.

Let $s$ be a local section of $h_x$ sending $x$ to 0 and defined on a neighborhood of $\gamma([0,\frac1n])$. We write $\gamma_0 = s \circ \gamma$ and $v_0 = \gamma_0(\frac1n)$. Therefore $h(x,\gamma_0(t)) = \gamma(t)$. So if $n=1$, the result is proved. Let us suppose it is proved for a certain $n$, and prove it for $n+1$.

By induction hypothesis, there is a path $\gamma_1$ in $T_{x_1}M$ such that $\gamma_1(0)=0$ and $h(x_1,\gamma_1(t)) = \gamma(\frac1n + \frac{n-1}{n}t)$. By definition of $f$, the fixed linear map $h'(x,v_0)$ is surjective, so there is a smooth path $\xi$ in $T_{(x,v_0)}TM$ (for instance with values in a fixed complement of $\ker h'(x,v_0)$) such that $h'(x,v_0)(\xi(t)) = \gamma_1(t)$ for $t \in [0,1]$, so
\[
h(x_1,\gamma_1(t)) = h \big( h(x,v_0),\gamma_1(t) \big) = h(x, v_0 + (T\tau_M \xi)(t)).
\]
Concatenating (and reparametrizing) the paths $\gamma_0$ and $t \mapsto v_0 + (T\tau \xi)(t)$ gives a path $\tilde{\gamma}$ satisfying the conclusion of the proposition.
\end{proof}

Let $h$ be a regular algebra on $M$ and $x \in M$. Let $N$ be a local transversal of the foliation at $x$ and let $Y \in \frakX(N)$. Define $f \colon N \to M$ by $f(y) = h(y,Y_y)$. Then $f$ is an immersion (if $N$ is small enough), since the rank of its differential is at least (and then exactly) the codimension of the foliation. Therefore if $N$ is small enough, $f(N)$ is a local transversal at $h(x,Y_x)$ and corresponds to ``sliding along the leaves'' of the foliation along the path $t \mapsto h(x,t Y_x)$, so it is an element of the holonomy groupoid of the foliation. Conversely, the lifting property of paths shows that any element in the holonomy groupoid is of this form.

\begin{thm}
A map in the linear holonomy of a regular monadic foliation has 1 as an eigenvalue.
\end{thm}

\begin{proof}
Fix an element of the holonomy group at $x$ given by a function $f$ as above.
That is, let $N$ be a local transversal, $Y \in \frakX(N)$ and $f \colon y \mapsto h(y,Y_y)$, with $f(x) = h(x,Y_x) = x$. We write $Y_x=v_0$. Then
\[
f'(x)y = h'(x,v_0)(y,Y'_x (y)).
\]
Differentiating the second axiom for algebras with respect to $(\dot{x},\dot{v})$ at $(x,v_0,\dot{x},\dot{v})$ gives
\[
{h_x}' \big( h'(x,v_0)(\dot{x},\dot{v}) \big) \circ h'(x,v_0) = {h_x}'(v_0+\dot{x}) \circ \pr_1.
\]
Fix $\dot{x} \in T_xN$ and suppose that 1 is not an eigenvalue of $f'(x) \in \calL(T_xN)$ (considered as an endomorphism via a fixed isomorphism $f'(x)T_xN \simeq T_xN$ coming from a ``small'' sliding). Then there is a $z \in T_xN$ such that $f'(x)z = z + v_0 + \dot{x}$. Therefore,
\[
h'(x,v_0) - \pr_1 \in \ker {h_x}'(v_0+\dot{x})
\]
so for all $y \in T_xN$, $f'(x)y - y \in \ker {h_x}'(v_0+\dot{x}) \subsetneq T_xN$, so $f'(x)-\id_{T_xN}$ is not surjective, so $1$ is an eigenvalue of $f'(x)$, which contradicts our assumption.
\end{proof}

This property of monadic foliations will prove useful in the next chapter for algebras on surfaces, where it implies that the characteristic multiplier of a periodic orbit has to be 1.

\section{Orientability of a monadic foliation and orientable cover}

A regular algebra is called (transversally) \define{orientable} if its induced distribution is.
For instance, the free algebra on $M$ is orientable if and only if $M$ is.
As an example, let $M$ be a non-orientable 1-dimensional manifold, for instance the boucle $\RR/\sim$ where then non-trivial identifications are $x\sim -x$ if and only if $|x|>1$. Note that this manifold is connected and second countable. Then the free algebra $h \colon T^2M \to TM$ induces a regular rank-1 foliation which is non orientable (a way to visualize this is to write $TM = (\RR \times \RR)/\sim$ where $(x,v) \sim (-x,-v)$ if $|x|>1$).

If $\calF$ is a regular foliation on the manifold $M$, then its orientation cover is the foliated two-folded covering $p_\oc \colon (M,\calF)_\oc \to (M,\calF)$ where $(M,\calF)_\oc$ is the foliated manifold on $\{(x,\calO_x) \mid x \in M \text{ and $\calO_x$ is an orientation of $T_x\calF$ } \}$ with obvious smooth structure, and the pullback foliation, which is orientable. For instance, the orientation cover of the ``U''-shape foliation on the torus is ``twice'' it (that is, two U's, so that the result is orientable). There is a similar notion of transverse orientation cover. For algebras of maximal rank, this gives the same result, since the tangent and normal bundles of the foliation are isomorphic.

If $h$ is a regular algebra on $M$, then it induces a regular foliation and by the general construction of Chapter~\ref{chap:tfmonad}, one can lift $h$ to an algebra $h_\oc$ on $M_\oc$. It is given  by
\begin{equation}
h_\oc \left( (x,\calO_x), v \right) = \left( h(x,v), \calO_{\calO_x,v} \right)
\end{equation}
where $\calO_{\calO_x,v}$ is the orientation of $T_{h(x,v)}\calF$ obtained from the orientation $\calO_x$ of $T_x\calF$ by following the path $[0,1] \to M, t \mapsto h(x,tv)$. Therefore, the foliation induced by $h_\oc$ is $\calF_\oc$.

Therefore, we can reduce the study of regular monadic foliations to orientable (or transversally orientable) ones.

%% file: chap5.tex
\chapter{Algebras of rank 1 and algebras on surfaces}
\label{chap:surf}

In this chapter, we study the special case of algebras of rank 1, and then of algebras on surfaces (necessarily of rank 1).
We will focus on regular algebras, which we can suppose orientable from the previous chapter. They are given by following the flow of a vector field for a certain time depending on a tangent vector. When this time is fiberwise linear, we give a characterization of these algebras. We give an example of an algebra with a non-linearizable ``time function'', but also give a sufficient condition for linearizability. Finally, we characterize morphisms between these algebras, which are the analogs of the conjugation of dynamical systems.

\section{Algebras of rank 1 and flows}

We begin with the study of regular algebras of rank 1, and we have seen that it is enough to study the orientable ones. It turns out that they have a rather special form.

\begin{prop}
\label{prop:rank1}
If $h \colon TM \to M$ is a regular orientable algebra of rank 1, then there exists a non-vanishing vector field $X$ with maximal flow $\phi$ and a smooth function $\alpha \in C^\infty(TM)$ with $\alpha(x,0)=0$ for any $x \in M$, unique among continuous functions with this property once $X$ is chosen, such that
\begin{equation}
h(x,v) = \phi_{\alpha(x,v)} (x)
\end{equation}
for any $(x,v) \in TM$.
\end{prop}

\begin{proof}
By hypothesis, the induced distribution $\calD$ is orientable and has rank~1, so is integrable in an orientable foliation of dimension~1, which is induced by a non-vanishing vector field $X$ whose maximal flow we denote by $\phi$.

Let $(x_0,v_0) \in TM$ and choose a distinguished chart around $h(x_0,v_0)$ and a time $t_0 \in \RR$ such that $\phi_{t_0}(x_0) = h(x_0,v_0)$. There are open neighborhoods $U$ of $(x_0,v_0)$ and $V$ of $t_0$ such that for $(x,v) \in U$ and $t \in V$, both $\phi_t(x)$ and $h(x,v)$ are in the domain of the distinguished chart and on the same plaque. Define $G \colon U \times V \to \RR$ by $G(x,v,t) = \phi_t(x) - h(x,v)$ (where in the right-hand side, we mean the projection on the leaf coordinate of what is written). Since $X$ does not vanish on $M$, $\frac{\partial G}{\partial t}(x_0,v_0,t_0) \colon \RR \to \RR$ is a linear isomorphism, so by the implicit function theorem, there are open neighborhoods of $(x_0,v_0)$ and $t_0$ included in $U$ and $V$ that we still denote by these letters, and a smooth function $\gamma \colon U \to V$ such that $\gamma(x_0,v_0)=t_0$ and for $(x,v) \in U$ and $t \in V$, $G(x,v,t) = 0$, that is, $h(x,v) = \phi_t(x)$, if and only if $t=\gamma(x,v)$.

If there is no periodic orbit, then the $\gamma$'s for different charts necessarily agree, so we can patch them to get $\alpha \in C^\infty(TM)$ and the proposition is proved.

In the general case, we proceed as follows.
For any $x \in M$, let $I_x$ be the domain of the maximal integral curve of $X$ through $x$ ($I_x$ is an interval containing 0). Then $\phi_{\ast}(x) \colon I_x \to \mathrm{Orb}(x)$ is a covering.
Therefore, $T_x M$ being simply connected, $h_x \colon T_xM \to \mathrm{Orb}(x)$, which is smooth because orbits are initial submanifolds, has a unique lift to a map $\alpha_x \colon T_x M \to I_x$ vanishing at 0.
We claim that this defines a smooth function $\alpha \colon TM \to \RR$.

Indeed, let $(x,v) \in TM$. The local result of the second paragraph of the proof applied at each $(x,tv)$ for $t \in [0,1]$ gives an open
neighborhood of $(x,tv)$ and a smooth function $\gamma$ on it. Without loss of generality, we can suppose that each of these open neighborhoods intersects $[0,v]$ (the segment joining $(x,0)$ to $(x,v)$ in $TM$) in a connected set.
Extract a finite cover of $[0,v]$ by these neighborhoods, say $U_i$'s, with $U_i \cap U_{i+1} \cap [0,v] \neq \varnothing$ for all $i$.
Then by induction on $i$ we see that the restriction of $\alpha$ to $U_i$ agrees with $\gamma_i$ (because of the uniqueness result in the implicit function theorem) for each $i$.
In particular, $\alpha$ is smooth at $(x,v)$.
\end{proof}

For algebras of this form, we have
\begin{equation}
A_x = {\alpha_x}'(0) \otimes X_x
\end{equation}
and more generally ${h_x}'(v) = {\alpha_x}'(v) \otimes X_{h(x,v)}$ if $(x,v) \in TM$.
The algebra axioms also translate into axioms on the map $\alpha \in C^\infty(TM)$. The first is of course that $\alpha(x,0)=0$, and the second is given in the following proposition.

\begin{prop}
\label{prop:timeAxioms}
The map $h$ defined by
\begin{equation*}
h(x,v) = \phi_{\alpha(x,v)} (x)
\end{equation*}
for any $(x,v) \in TM$, where $\phi$ is the maximal flow of the vector field $X$ and $\alpha \in C^\infty(TM)$ with $\alpha(x,0)=0$, is an algebra if and only if
\begin{align}
\alpha(x,v+\lambda X_x) &= \alpha(x,v) \quad\text{and}\\
\alpha \big( \phi_{\alpha(x,v)} (x), {\phi_{\alpha(x,v)}}' (x) \dot{x} \big) &= \alpha(x,v+\dot{x}) - \alpha(x,v) \label{eq:alpha}
\end{align}
for any $x \in M$, $v, \dot{x} \in T_xM$ and $\lambda \in \RR$.
\end{prop}

\begin{proof}
From the general property $h(x,v+\calD_x)=h(x,v)$, we deduce that $\alpha(x,v+\lambda X_x) \in \alpha(x,v) + P_x \ZZ $ where $P_x$ is the period of $x$, or 0 if the orbit of $x$ is not periodic. By continuity of $\alpha$, we deduce $\alpha(x,v+\lambda X_x) = \alpha(x,v)$.

For the second relation, we have in charts
$h'(x,v)(\dot{x},\dot{v}) = \alpha'(x,v)(\xi) X_{h(x,v)} + {\phi_{\alpha(x,v)}}'(x) \dot{x}$,
therefore, using the first relation, the second algebra axiom reads
\[
\phi_{ \alpha \big( h(x,v) , \phi_{\alpha(x,v)}'(x)\dot{x} \big) } \big( \phi_{\alpha(x,v)}(x) \big) = \phi_{\alpha(x,v+\dot{x})} (x).
\]
From the morphism property of flows, we deduce
\[
\alpha \big( h(x,v) , {\phi_{\alpha(x,v)}}'(x)\dot{x} \big) \in \alpha(x,v+\dot{x}) - \alpha(x,v) + P_x \ZZ.
\] 
Again, this expression is continuous in $v$ and we have equality for $v=0$, so we can ignore the $P_x\ZZ$.

Conversely, if $\alpha$ satisfies these equations, then $h$ is an algebra.
\end{proof}

The first condition says that we can consider $\alpha$ as a function on the transverse bundle ($\alpha$ is ``semibasic''), and the second is a sort of invariance property, making $\alpha$ a kind of nonlinear transverse structure of the foliation. A consequence is that for a regular algebra of rank 1, the maps ${h_x}'(v)$ all have rank 1.

If $\alpha$ is a 1-form, then the situation is clearer. Recall that if $X \in \frakX(M)$, a differential form $\omega \in \Omega(M)$ is called \define{$X$-semibasic} if $i_X \omega = 0$ and \define{$X$-basic} if furthermore $i_X d\omega = 0$ or equivalently $\calL_X \omega = 0$. These notions depend only on the foliation generated by $X$, and more generally, if $\calF$ is a foliation, a form is called $\calF$-(semi)basic if it is $X$-(semi)basic for any vector field $X$ tangent to $\calF$.
Basic forms of foliations are studied in~\cite{tond} and~\cite{cama}.

\begin{prop}
Let $M$ be a manifold, $X \in \frakX(M)$ be a complete vector field on $M$, and $\alpha \in \Omega^1(M)$ be a 1-form on $M$. Let $\phi$ be the maximal flow of $X$.
Then the map
\begin{align*}
h: TM & \to M\\
(x,v) & \mapsto \phi_{\alpha_x(v)} (x)
\end{align*}
is an algebra if and only if $\alpha$ is $X$-basic.
\end{prop}

\begin{proof}
If $\alpha \in \Omega^1(M)$, then the relations of the previous proposition become
\begin{align*}
i_X \alpha &= 0\quad \text{and}\\
\big( {\phi_{\alpha_x(v)}}^*\alpha\big)_x (\dot{x}) &= \alpha_x(\dot{x}).
\end{align*}
In particular, $\alpha_x$ is either 0 or non-vanishing on any orbit of $X$. If it is 0, there is nothing to prove. If it is non-vanishing, given any $t \in \RR$, there is a $v \in T_xM$ such that $\alpha_x(v)=t$, so that
\[
\big( {\phi_t}^*\alpha\big)_x (\dot{x}) = \alpha_x(\dot{x})
\]
for any $t \in \RR$, so $\calL_X \alpha = 0$.
\end{proof}

As an example, if $\omega \in \Omega^2(M)$ is a presymplectic form and $X$ is an $\omega$-presymplectic vector field, then $i_X \omega$ is an $X$-basic 1-form. This shows that ``presymplectizable'' flows are monadic.

\begin{rmk}
One can construct similar ``reducible'' examples of algebras of higher rank by using commuting flows of independent vector fields and independent common basic 1-forms.
\end{rmk}

\subsection{An algebra inducing a foliation with no non-vanishing basic form}

One can ask if all algebras of rank 1 are of the form of the previous proposition. This is not the case, even on surfaces. To construct an example, we first need the following compatibility result:

The conditions of Proposition~\ref{prop:timeAxioms} on $\alpha$ determine it on the whole orbit of a point once we know it at one point. The converse statement is also true. Namely, suppose that for a given $x \in M$, we have
$\alpha(\phi_{\alpha(x,v)}(x), {\phi_{\alpha(x,v)}}'(x)\dot{x}) = \alpha(x,v+\dot{x}) - \alpha(x,v)$
for all $v, \dot{x} \in T_xM$.
We will show that $\alpha$ satisfies the required condition~\eqref{eq:alpha} on the whole orbit of $x$.
Let $y = \phi_{\alpha(x,v_0)}(x)$. Then
\begin{equation}\label{eq:comp1}
\alpha(y,w) = \alpha(\phi_{\alpha(x,v_0)}(x),w) = \alpha(x,v_0+{{\phi_{\alpha(x,v_0)}}'(x)}^{-1}w) - \alpha(x,v_0)
\end{equation}
and similarly for $\alpha(y,w+\dot{y})$, so that
\begin{equation}\label{eq:comp2}
\alpha(y,w+\dot{y}) - \alpha(y,w) = \alpha(x,v_0 + {{\phi_{\alpha(x,v_0)}}'(x)}^{-1} (w+\dot{y})) - \alpha(x,v_0 + {{\phi_{\alpha(x,v_0)}}'(x)}^{-1} w).
\end{equation}
On the other hand,
\begin{align*}
\alpha(\phi_{\alpha(y,w)}(y),{\phi_{\alpha(y,w)}}'(y)\dot{y})
&=
\alpha(\phi_{\alpha(y,w)+\alpha(x,v_0)}(x), {\phi_{\alpha(y,w)+\alpha(x,v_0)}}'(x) \circ {{\phi_{\alpha(x,v_0)}}'(x)}^{-1} \dot{y})\\
&=
\alpha(\phi_{\alpha(x,v_0+{{\phi_{\alpha(x,v_0)}}'(x)}^{-1}w)}(x),\\
& \qquad\qquad {\phi_{\alpha(x,v_0+{{\phi_{\alpha(x,v_0)}}'(x)}^{-1}w)}}'(x) \circ {{\phi_{\alpha(x,v_0)}}'(x)}^{-1} \dot{y})\\
&=
\alpha(x,v_0+{{\phi_{\alpha(x,v_0)}}'(x)}^{-1}w + {{\phi_{\alpha(x,v_0)}}'(x)}^{-1}\dot{y})\\
& \qquad\qquad - \alpha(x,v_0+{{\phi_{\alpha(x,v_0)}}'(x)}^{-1}w)
\end{align*}
which has the same right hand side as~\eqref{eq:comp2}, which proves the result.

Our example of algebra will induce the foliation given by the vector field $X = \id_{\RR^2}$. Its flow is given by $\phi_t(x) = e^t x$ for $x \in \RR^2$ and $t \in \RR$.
It has no nonzero basic 1-form. Indeed, such a form would be closed hence exact, so is the differential of a basic function, but the origin is in the closure of all the orbits, so the basic functions are constant. We want to construct an algebra $h$ given by $h(x,v) = \phi_{\alpha(x,v)}(x)$, and we look for one with symmetry of rotation. By the above result, it is therefore enough to define $f(u)=\alpha( \vect{1}{0}, \vect{0}{u})$ for $u \in \RR$. This function $f$ should be a diffeomorphism of $\RR$, say increasing, vanishing at 0. The only condition to check is that the algebra $h$ is smooth at the origin.

If $x>0$, then $h(\vect{x}{0},\vect{0}{v}) = \exp \left( f(f^{-1}(\ln x) + \frac{v}{x}) \right) \vect{1}{0}$, so we want that $x \to 0^+$ and $v$ bounded imply $f(f^{-1}(\ln x) + \frac{v}{x})) \to -\infty$. This will be the case if $f$ increases slowly enough near $-\infty$, for instance $f(x) = -c \ln(-x)$ with $c \in (0,1)$. If we choose $c=\frac12$, the algebra $h$ takes the nice form
\begin{equation}
h(x,v) = \frac{x}{\sqrt{1 - x \wedge v}}
\end{equation}
when $x \wedge v < \frac12$, for instance.

\section{Algebras on surfaces}
\label{sec:surf}

We have seen in the previous example that not all orientable algebras of rank 1 are given by a 1-form. However, if the induced foliation has a non-vanishing basic 1-form, then it is possible to rescale the vector field so that the algebra be given by a 1-form. This is the object of the following proposition. It uses the following lemma.

\begin{lem}
Let $h$ be an algebra on a manifold $M$ given by $h(x,v) = \phi_{\alpha(x,v)} (x)$ where $\phi$ is a flow of a vector field $X \in \frakX(M)$ and $\alpha \in C^\infty(TM)$ with $\alpha(x,0)=0$. If $\alpha(x_0,\cdot)$ is linear for some $x_0 \in M$ and $y_0 \in L_{x_0}$, then $\alpha(y_0,\cdot)$ is also linear.
\end{lem}

\begin{proof}
There is a $v \in T_{x_0}M$ such that $h(x_0,v)=y_0$, so that Equation~\eqref{eq:alpha} reads
\[
\alpha \big( y_0, {\phi_{\alpha(x_0,v)}}'(x_0)(\dot{x}) \big) = \alpha(x_0,v+\dot{x}) - \alpha(x_0,v) = \alpha(x_0,\dot{x})
\]
for any $\dot{x} \in T_{x_0}M$, that is, $\alpha ( y_0, \cdot ) = \alpha(x_0,\cdot) \circ {{\phi_{\alpha(x_0,v)}}'(x_0)}^{-1}$, which is linear.
\end{proof}

\begin{prop}
\label{prop:form}
If a regular algebra of rank 1 on a surface induces an orientable foliation which has a non-vanishing basic 1-form, then it can be written as
\[
(x,v) \mapsto \phi_{\alpha_x(v)}(x)
\]
where $\phi$ is the maximal flow of a non-vanishing vector field $X$ and $\alpha$ is a non-vanishing $X$-basic 1-form.
\end{prop}

The proof relies on the following lemma concerned with the flow of a rescaled vector field.

\begin{lem}
Let $M$ be a manifold, $x \in M$, $X \in \frakX(M)$ and $f \in C^\infty(M)$ both non-vanishing. If $I$ an open interval of $\RR$ and $\mu, \nu \in C^\infty(I,\RR)$ are equal at some $t_0 \in I$, then
\[ 
\phi^{fX}_{\mu(t)} (x) = \phi^{X}_{\nu(t)} (x) \text{ on $I$}
\qquad \text{if and only if} \qquad
\dot{\nu}(t) = f \left( \phi^{fX}_{\mu(t)} (x) \right) \dot{\mu}(t) \text{ on $I$}
\]
and in that case, $\mu$ and $\nu$ determine each other.
\end{lem}

\begin{proof}
This is just differentiation in $t$.
\end{proof}

We also need the following result from the theory of foliations. If $X$ is a vector field on a manifold $M$, then a vector field $Y$ is called \define{$X$-foliate} if $[X,Y] \in C^\infty(M) X$, or equivalently $(\phi_t)_* Y_x - Y_{\phi_t(x)} \in \RR X_{\phi_t(x)}$ for any $x \in M$ and $t \in \RR$. See for instance~\cite{cama} or~\cite{tond}.

\begin{lem}
If a non-vanishing vector field on a surface has a non-vanishing basic 1-form, then it also has a foliate vector field nowhere tangent to it.
\end{lem}

\begin{proof}
Let $X$ be a non-vanishing vector field on the surface $M$ and $\alpha$ be a non-vanishing $X$-basic 1-form. Define $Y$ to be the vector orthogonal to $X$ (for some Riemannian metric) such that $\alpha_x(Y_x) = 1$. Then $\alpha({\phi_t}_* Y) = ({\phi_t}^*\alpha)({\phi_t}_* Y) = \alpha(Y) = 1$, so ${\phi_t}_* Y_x - Y_{\phi_t(x)} \in \ker \alpha_{\phi_t(x)} = \RR X_{\phi_t(x)}$ for any $x \in M$, and $Y$ is nowhere tangent to $X$.
\end{proof}

\begin{proof}[Proof of Proposition~\ref{prop:form}]

Let $h$ be such an algebra on the surface $M$. By Proposition~\ref{prop:rank1}, there is a non-vanishing vector field $Y \in \frakX(M)$ and a function $\beta \in C^\infty(TM)$ such that $h(x,v) = \phi_{\beta(x,v)}^Y (x)$ for any $(x,v) \in TM$.

If we rescale $Y$ by a non-vanishing function $f \in C^\infty(M)$, and define $X=fY$, then by the lemma, ${\beta_x}'(v) = f(h(x,v)) {\alpha_x}'(v)$
and we want that $\alpha_x$ be linear, that is, ${\alpha_x}'(v)={\alpha_x}'(0)$, which means $f(x) {\beta_x}'(v) = f(h(x,v)) {\beta_x}'(0)$.
Let $Z$ be a foliate vector field  nowhere tangent to $Y$, and define $f(x) = {\beta_x}'(0)Z_x$. Then
\[
f(h(x,v)) = {\beta_{h(x,v)}}'(0) Z_{h(x,v)} = {\beta_{h(x,v)}}'(0) \circ {\phi_{\beta(x,v)}}'(x) Z_x 
\]
(the latter equality is because $Z_{h(x,v)} = {\phi_{\beta(x,v)}}'(x) Z_x + \lambda Y_{h(x,v)}$ and $Y_{h(x,v)} \in \ker {\beta_{h(x,v)}}'(0)$, and differentiating fiberwise the algebra axiom for $\beta$, we obtain
\[
{\beta_{h(x,v)}}'(0) \circ {\phi_{\beta(x,v)}^Y}'(x) = {\beta_x}'(v),
\]
So we have to prove that
\[
{\beta_x}'(0) (Z_x) .  {\beta_x}'(v) = {\beta_x}'(v)(Z_x) . {\beta_x}'(0).
\]
Since the algebra has rank 1, $\ker {\beta_x}'(0) = \ker {\beta_x}'(v) = \calD_x$ (paragraph following the proof of Proposition~\ref{prop:timeAxioms}), and since $M$ is a surface, $\calD_x \oplus \RR Z_x = T_xM$. If we apply the above relation to any argument $w \in T_xM$, we can suppose as well $w = k Z_x$, and the two sides are seen to be equal.
\end{proof}

\section{Morphisms of algebras of rank 1}

Recall that in general vector fields cannot be pulled back or pushed forward. A common generalization of both notions is that of $f$-related vector fields: if $f \colon M \to N$ and $X \in \frakX(M)$, $Y \in \frakX(N)$, then $X$ and $Y$ are \define{$f$-related} if $Tf \circ X = Y \circ f$. This implies that their flows are conjugated by $f$, that is, $f \circ \phi^X_t = \phi^Y_t \circ f$ where defined, and the converse is true, as seen by differentiating the latter equation in $t$.

\begin{prop}
If $f \colon M \to N$ is a morphism from a regular algebra given as above by $(X,\alpha)$ to a regular algebra given by $(Y,\beta)$, then we can rescale $Y$ to $\tilde{Y}$so that $X$ and $\tilde{Y}$ be $f$-related, and
\[
\tilde{\alpha}(x,v) = \beta(f(x),f'(x)v)
\]
for any $(x,v) \in TM$. Conversely, an $f$ realizing these two conditions is a morphism. If $\tilde{\alpha}$ is linear and $f$ is a submersion, or if $\beta$ is linear, then both $\tilde{\alpha}$ and $\beta$ are linear and $\tilde{\alpha} = f^* \beta$.
\end{prop}

\begin{proof}
Let $\phi$ denote the flow of $X$ and $\psi$ that of $Y$. That $f$ is an algebra morphism reads
\[
\psi_{\beta(f(x),f'(x)v)} (f(x)) = f(\phi_{\alpha(x,v)} (x)).
\]
Differentiating with respect to $v$ at $v=0$ gives
\[
{\beta_{f(x)}}'(0) \circ f'(x) \otimes Y_{f(x)} = {\alpha_x}'(0) \otimes f'(x)(X_x).
\]
Since the algebra given by $(X,\alpha)$ is regular, there is a $v \in T_xM$ such that ${\alpha_x}'(0)v \neq 0$, so we define $g \in C^\infty(M)$ by
\[
g(x) = \frac{{\beta_{f(x)}}'(0) \circ f'(x)v}{{\alpha_x}'(0)v}
\]
for any such $v$ (the value is independent of $v$ since these forms have the same kernel).
Define $\tilde{X} = gX$. Then $Y_{f(x)} = f'(x)(\tilde{X}_x)$, which means that $\tilde{X}$ and $Y$ are $f$-related. Also, the corresponding function $\tilde{\alpha}$ is such that $\psi_{\beta(f(x),f'(x)v)} (f(x)) = f(\tilde{\phi}_{\tilde{\alpha}(x,v)} (x))$, so $\beta(f(x),f'(x)v) = \tilde{\alpha}(x,v)$ from the discussion before the proposition.

The converse and the second part of the proposition are obvious.
\end{proof}